\documentclass[a4paper,11pt]{article}
\usepackage[english]{babel}

\usepackage{amsxtra}
\usepackage{amssymb}
\usepackage{amsmath}
\usepackage{amsthm}
\usepackage{color}
\usepackage{esint}

\usepackage{fancyhdr}

\usepackage{hyperref}

\newtheorem{prop}{Proposition}[section]
\newtheorem{define}[prop]{Definition}
\newtheorem{lemma}[prop]{Lemma}
\newtheorem{theo}[prop]{Theorem}
\newtheorem{coro}[prop]{Corollary}

\numberwithin{equation}{section}

\theoremstyle{remark}
\newtheorem{rmq}{Remark}

\newcommand{\wE}{\widetilde{E}}
\newcommand{\dun}{\partial_1}
\newcommand{\n}{\nabla}

\newcommand{\R}{\mathbb{R}}
\newcommand{\N}{\mathbb{N}}
\newcommand{\Z}{\mathbb{Z}}

\title{Small energy traveling waves for the Euler-Korteweg system}
\author{Corentin Audiard
\footnote{Sorbonne Universit\'es, UPMC Univ Paris 06, UMR 7598, Laboratoire Jacques-Louis Lions, 
F-75005, Paris, France }
\footnote{CNRS, UMR 7598, Laboratoire Jacques-Louis Lions, F-75005, Paris, France}}

\begin{document}
\maketitle

\begin{abstract}
We investigate the existence and properties of traveling waves for the Euler-Korteweg system 
with general capillarity and pressure. 
Our main result is the existence in dimension two of waves with arbitrarily small energy. 
They are obtained as minimizers of a modified energy with fixed momentum. The 
proof follows various ideas developed for the Gross-Pitaevskii equation (and more generally 
nonlinear Schr\"odinger equations with non zero limit at infinity). Even in the Schr\"odinger
case, the fact that we work with the hydrodynamical variables and a general pressure law
both brings new difficulties and some simplifications.
Independently, in dimension one we prove that the criterion for the linear instability of traveling waves 
from \cite{Benzoni5} actually implies nonlinear instability.
\end{abstract}

\renewcommand{\abstractname}{R\'esum\'e}
\begin{abstract}
On \'etudie les ondes progressives des \'equations d'Euler-Korteweg pour des lois de 
capillarit\'e et pression g\'en\'erales.
Le principal r\'esultat est l'existence en dimension $2$ d'ondes d'\'energie arbitrairemet petite.
Elles sont obtenues comme minimiseurs d'une \'energie modifi\'ee
\`a moment fix\'e. La preuve suit plusieurs id\'ees d\'evelopp\'ees pour les \'equations de 
Schr\"odinger non lin\'eaires avec limite non nulle \`a l'infini. M\^eme dans ces cas, 
le fait de travailler en variables hydrodynamiques apporte de nouvelles difficult\'es, mais aussi 
quelques simplifications.
Ind\'ependamment, on montre en dimension un que le crit\`ere d'instabilit\'e lin\'eaire des ondes 
progressives de \cite{Benzoni5} implique en fait l'instabilit\'e non lin\'eaire.
\end{abstract}

\tableofcontents
\section{Introduction}
The Euler-Korteweg system is a modification of the usual Euler equations for compressible fluids 
that includes capillary effects. 
Mathematically it reads in dimension $d$ as the following system 
of $d+1$ equations combining the conservation of mass and momentum
\begin{equation}\label{EK}
\left\{
\begin{array}{ll}
\partial_t\rho+\text{div}(\rho u)=0,\\
\partial_t u+u\cdot \nabla u+\nabla g(\rho)=\nabla \bigg(K(\rho)\Delta\rho+\frac{1}{2}K'(\rho)
|\nabla\rho|^2\bigg),
\end{array}
\right.
x\in \R^d.
\end{equation}
The variables $\rho$ and $u$ are the density and speed of the fluid, the right hand side of the 
second line is the so called capillary tensor. The functions $K,g$ are defined on 
$\R^{+*}$ and are supposed to be smooth and positive. For the equations to make sense, it is necessary that 
$\rho>0$ a.e .
\\
The Korteweg tensor was first derived in the work of Dunn and Serrin \cite{DunSer} for 
models of phase transition, however the equations can appear in very various 
settings, from water waves (see \cite{BreschDesj}) to quantum hydrodynamics.\\
When $u$ is potential (the irrotational case)
\eqref{EK} has a hamiltonian structure : indeed if we write $u=\nabla \phi$, the second line of 
\eqref{EK} rewrites 
\begin{equation*}
\partial_t \phi+\frac{|\nabla\phi|^2}{2}+g(\rho)=K\Delta\rho+\frac{1}{2}K'(\rho)|\nabla \rho|^2,
\end{equation*}
For $G$ a primitive of $g$, we define the energy 
\begin{equation}\label{energy}
E(\rho,\phi)=\int \frac{K(\rho)|\nabla\rho|^2+\rho|\nabla \phi|^2}{2}+G(\rho)dx, 
\end{equation}
then \eqref{EK} reads 
\begin{equation}\label{hamEK}
\left\{
\begin{array}{ll}
\displaystyle \partial_t\rho -\frac{\delta E}{\delta \phi}=0,\\
\displaystyle \partial_t\phi + \frac{\delta E}{\delta \rho}=0.
\end{array}\right.
\end{equation}
In particular, for $(\rho,\phi)$ a solution with enough integrability and smoothness, 
$E(\rho,\phi)(t)$ is conserved. One can also check formally the conservation of momentum: if 
$\displaystyle \lim_{|x|\rightarrow \infty}\rho=\rho_0\in \R^{+*}$,
\begin{equation}\label{momentum}
\frac{d\vec{P}}{dt}:=\frac{d}{dt}\int_{\R^d} (\rho-\rho_0)\nabla \phi=0.
\end{equation}
Concerning the analysis of well-posedness, it was observed in \cite{BDD} that for smooth solutions 
without vacuum \eqref{EK} is equivalent to a quasi-linear degenerate Schr\"odinger equation. Due 
to this very nonlinear structure, the analysis of the Cauchy problem is quite involved. If 
$\overline{\rho},\overline{u}$ is a reference smooth solution, local well-posedness for 
$(\rho_0,u_0)\in (\overline{\rho}(0)+H^{s+1})\times (\overline{u}+H^s),\ s>d/2+1$ was obtained in 
\cite{Benzoni1}. The energy of the system allows to control at best $(\rho-\overline{\rho},
u-\overline{u})$ in $H^1\times L^2$, so in any dimension global well-posedness has remained mostly an open 
problem.\vspace{2mm}\\
In the special case $K=\kappa/\rho$, with $\kappa$ a positive constant and 
$u=\nabla \phi$ is irrotational, up to some rescaling there exists a formal correspondance with the 
nonlinear Schr\"odinger equation 
\begin{equation}\label{NLS}
i\partial_t\psi+\Delta \psi =g(|\psi|^2)\psi,
\end{equation}
through the \emph{Madelung transform} $(\rho,\nabla \phi)\mapsto \psi:=\sqrt{\rho}e^{i\phi}$, 
introduced in \cite{Madelung} (for more details see the review article \cite{CDS}). We will not dwell 
upon it but only mention that the nonlinearity $g$ in \eqref{NLS} becomes the pressure 
term, and if $\rho$ vanishes the transform becomes singular. 
Antonelli and Marcati \cite{AntMarc2} managed to exploit this correspondance 
in order to pass from global solutions of NLS (whose existence is standard, see the 
reference book \cite{Cazenave}) 
to global weak solutions of \eqref{EK}. In general such solutions admit vacuum and one can not 
hope to deduce uniqueness from such arguments.
In the special case $g(\rho)=\rho-1$, \eqref{EK} corresponds to the Gross-Pitaevskii equation 
which has received a lot of attention over the last fifteen years. In particular, global dispersive 
solutions of \eqref{NLS} were constructed in \cite{GNT3}. Such results were used 
to construct global \emph{unique} solutions of \eqref{EK} for small irrotational data by the author and 
B.Haspot in \cite{AudHasp}. \\
The result was later extended by the same authors in \cite{AudHasp2} for general $K$, $g$ and $d\geq 3$: 
for initial data near the constant state $(\rho_0,0)$ with the stability condition $g'(\rho_0)>0$, the 
solution is global and converges to a solution of the linearized equation near $(\rho_0,0)$ (in other words 
it scatters). The price to pay for this generalization is the necessity to work with much 
smoother functions, basically: $\rho-\rho_0\in H^{50}$. The idea behind this result is that 
the Euler-Korteweg equations \eqref{EK} and the Gross-Pitaevskii equation share the same linearized system (near 
$(\rho,u)=(\rho_0,0)$, resp. $\psi=1$) 
so that the same small data technics from the field of dispersive equations can be used. 

A natural question is then wether such an analogy is still true for \emph{nonlinear} phenomena
and in particular the existence for traveling waves which is known for a large class of 
nonlinear Schr\"odinger equations. This article gives a partial positive answer: our main result 
(theorem \ref{maintheo})
is the existence of small traveling waves in dimension $2$. Before turning to a precise 
statement, let us give some background about this issue.\\
The existence of \emph{planar} traveling waves, that is solutions of the form 
$(\rho(x_1-ct),u(x_1-ct))$ is a simpler problem as in this case \eqref{EK} can be reduced to a system of 
two ODEs. Due to the hamiltonian nature of the equations 
these ODEs are integrable by quadrature. If $g$ is not monotone (for example with a Van der 
Waals pressure law) all three types of interesting solutions exist (homoclinic, heteroclinic 
and periodic). In dimension one, the stability/unstability of such solutions is related to the 
notion of moment of instability from the seminal paper \cite{GSS} of Grillakis-Shatah-Strauss. 
In the absence of global well-posedness result, only conditional stability was derived for the 
corresponding traveling waves. On this topic, we offer a small contribution with theorem 
\ref{nonlininsta} which states that failure of the stability criterion from \cite{BDD2} implies 
nonlinear instability. For more details on planar traveling waves 
we refer to the rich review article \cite{Benzoni5}.
\paragraph{The existence of localized solitary waves} In dimension larger than $1$ the existence 
of \emph{localized} traveling waves that depend on $x-\vec{c}t$, with $\vec{c}$ the direction and speed 
of propagation, has been so far an open problem. Our main result is the existence of small energy traveling waves 
(see theorem \ref{maintheo}). The interest is twofold: besides giving global solutions 
to \eqref{EK}, the existence of arbitrarily small 
solitary waves in dimension $2$ is in strong contrast with the 
scattering of small solutions in dimension $\geq 3$. Note that while our results apply 
to general $K,\ g$, in the special case $K(\rho)=\kappa/\rho$, $g(\rho)=\rho-1$, the existence of 
solitary waves to \eqref{EK} might be deduced thanks to the Madelung transform from the existence 
of non-vanishing solitary waves to \eqref{NLS}.
However if $K$ is not proportional to $1/\rho$ (even $K\equiv$ constant), new difficulties appear due to the 
quasilinear nature of \eqref{EK}. \\
Concerning the expected range of speeds, the linearization of the Euler equations (without 
capillary terms) near $\rho=\rho_0,\ u=0$ is
\begin{equation*}
\left\{
\begin{array}{ll}
\partial_t\rho+\rho_0\text{div}u=0,\\
\partial_t u+g'(\rho_0)\nabla \rho=0.
\end{array}
\right.
\end{equation*}
If $g'(\rho_0)>0$, $\rho$ satisfies the wave equation $\partial_t^2\rho-\rho_0
g'(\rho_0)
\Delta \rho=0$, with the so-called \emph{sound speed} $c_s(\rho_0)=
\sqrt{\rho_0 g'(\rho_0)}$. By analogy with the Gross-Pitaevskii case, we expect that traveling waves with 
limit at infinity $(\rho_0,0)$ can only exist for subsonic speeds $|\vec{c}|\leq c_s$. Obviously,
the direction of the speed does not matter, thus from now on we restrict ourselves to $\vec{c}=c\,\vec{e_1}$.

\subparagraph{Some results on \eqref{NLS} with nonzero condition at infinity} If $g(1)=0$, 
a natural problem is the construction of solitary waves such that 
$\lim_{|x|\rightarrow \infty}|\psi|=1$. The case of the Gross-Pitaevskii equation $g(\rho)=\rho-1$ 
has attracted a lot of attention since the series of papers of Roberts and al 
\cite{Jones2}\cite{Jones}. Their formal and numerical computations brought a number of conjectures on
the existence of branches of solitary waves with speeds $c$ covering the subsonic range 
$(0,\sqrt{2})$ (the number $\sqrt{2}$ is related to $\sqrt{1\cdot g'(1)}=1$ after some rescaling), their 
stability and limit in the transonic regime. In dimension $2$ traveling waves were constructed for any 
$|c|$ small enough by B\'ethuel and Saut (\cite{BethSaut}, $98$) with a moutain pass argument. More recently 
they used with P. Gravejat in \cite{BGS} a constrained minimization method, that we will also follow.
To shortly describe it, let us introduce the momentum and energy
\begin{equation*}
\text{(momentum) }(\vec{P}_{NLS}(\psi)=\frac{1}{2}\text{Re}\int_{\R^d} i\nabla \psi(\overline{\psi}-1)dx.
\end{equation*}
\begin{equation*}
\text{(energy) }E_{NLS}(\psi)=\frac{1}{2}\int_{\R^d}|\nabla\psi|^2+G(|\psi|^2)dx
\end{equation*}
where $G$ is a primitive of $g$. In the Gross-Pitaevskii case, we simply have $G=(|\psi|^2-1)^2/2$.
These two quantities are formally conserved by the flow, also it is not hard to check formally
$$c\delta P_{NLS,1}(\psi)=\delta E_{NLS}(\psi)\Leftrightarrow -ic\partial_1\psi+\Delta\psi=(|\psi|^2-1)\psi.$$ 
It is thus tempting to construct 
solitary waves as minimizers of the energy with $P_{NLS,1}(\psi)=p$ fixed. 
However, it is a bit tedious
to give a functional framework where both $E_{NLS}$ and $\vec{P}_{NLS}$ make sense, and the 
existence of a lifting of $\psi$ on subsets of $\R^d$, while extremely useful, raises significant 
topological difficulties. Finally, in this approach the speed $c$ is only obtained as a Lagrangian 
multiplier, which precludes to reach the whole range $c\in (0,\sqrt{2})$. Nevertheless in \cite{BGS} 
the authors proved the existence of a branch of solutions parametrized by the momentum $p\in (p_0,\infty)$ 
in dimension $2$ and $3$ ($p_0=0$ in dimension $2$, $>0$ in dimension $3$). 
With an alternative approach, Maris \cite{Maris} obtained the existence of traveling waves for 
the full range $c\in (0,\sqrt{2})$ in dimension $\geq 3$ for a class of equations more general 
than Gross-Pitaevskii. The proof relied on the minimization of energy with a more subtle 
constraint based on a Pohozaev type identity.
Finally, the construction of solitary waves by minimization with fixed momentum was recently 
improved by Maris and Chiron \cite{ChirMar}, giving the precompactness of minimising sequences 
(which is a classical ingredient for orbital stability). 
\begin{rmq}
Note that in the case $K=1/\rho$, energy and momentum conservation for $\psi$ exactly correspond to 
energy and momentum conservation in \eqref{energy},\eqref{momentum} with $\psi=\sqrt{\rho}e^{i\phi/2}$. 
Indeed 
% \footnote{
% The Madelung transform $\psi\mapsto (\rho,\nabla\phi)$ has thus the remarkable feature of being a
% canonical transformation.}
\begin{equation*}
E_{NLS}(\psi)=\frac{1}{2}\int_{\R^d}\frac{1}{4\rho}|\nabla\rho|^2+\frac{\rho}{4}|\nabla\phi|^2+
G(\rho)dx,\
\vec{P}_{NLS}=\frac{1}{4}\int \rho \nabla \phi=\frac{1}{4}\int (\rho-1)\nabla\phi\,dx.
\end{equation*}
\end{rmq}
\subparagraph{Rescaling, modified energy and main result}
We construct solutions of \eqref{EK} of 
the form 
$(\rho(x_1-ct,x_2)$, $u(x_1-ct,x_2))$, the system of partial differential equations to solve is 
\begin{equation}\label{TWEKnonrescal}
\left\{
\begin{array}{ll}
\displaystyle 
-c\partial_1\rho +\text{div}(\rho\nabla \phi)=0,\\
\displaystyle 
-c\partial_1\phi+\frac{|\nabla \phi|^2}{2}-K\Delta\rho-\frac{K'|\nabla\rho|^2}{2}+
g(\rho)=0.
\end{array}
\right.
\end{equation}

We will focus on the existence of localized traveling waves near the constant state $(\rho_0,0)$ 
with $g(\rho_0)=0$, $g'(\rho_0)>0$. We use the following rescaling :
\begin{eqnarray*}
(\rho,\phi)=\bigg(\rho_0\rho_r\bigg(\sqrt{\frac{g'(\rho_0)}{\rho_0)}}x\bigg),\ 
\phi_r\bigg(\sqrt{\frac{g'(\rho_0)}{\rho_0)}}x\bigg)\bigg),\ 
K_r(\rho_r)=\frac{K(\rho_0\rho_r)}{\rho_0},\\
g_r(\rho_r)=\frac{g(\rho_0\rho_r)}{g'(\rho_0)\rho_0},\ 
c_r=\frac{c}{\sqrt{\rho_0g'(\rho_0)}}.
\end{eqnarray*}
Then \eqref{TWEKnonrescal} is equivalent to 
\begin{equation}\label{TWEK}
\left\{
\begin{array}{ll}
\displaystyle 
-c_r\partial_1\rho_r+\text{div}(\rho_r\nabla \phi_r)=0,\\
\displaystyle 
-c_r\partial_1\phi_r+\frac{|\nabla \phi_r|^2}{2}-K_r\Delta\rho_r
-\frac{K_r'|\nabla\rho_r|^2}{2}+
g_r(\rho)=0.
\end{array}
\right.
\end{equation}
Of course, the point of this rescaling is that the constant state is now $1$, $g_r'(1)=1$ and the sound speed is 
$\sqrt{1g_r'(1)}=1$. From now on we drop the $r$ index and work on the rescaled system.

\noindent If we define the scalar momentum $P$ as
\begin{equation*}
P(\rho,\phi)=\int_{\R^d}(\rho-1)\partial_1\phi dx,
\end{equation*}
our starting point is that similarly to NLS, \eqref{TWEK} can be recast as *
$$c\delta P(\rho,\phi)=\delta E(\rho,\phi).$$
The scalar momentum $P$ is well-defined on\footnote{In dimension 
$2$ the space $\dot{H}^1$ requires a bit of cautiousness, see definition \ref{defH1}} 
$\mathcal{H}:=\{(\rho,\phi)\in (1+H^1)\times \dot{H}^1\}$. On the other hand, the energy has two 
flaws: depending on $G$ it may not make sense for general $(\rho,u)\in \mathcal{H}$, and even in the 
simple case $G=(\rho-1)^2/2$ it satisfies no coercive 
inequality $E(\rho,\phi)\gtrsim \|(\rho,\phi)\|_{\mathcal{H}}^2$.
Since we are interested in the  regime $|\rho-1|<<1$, the remedy is to work with a 
modified energy that we define now. We fix $\chi\in C^\infty(\R^+)$ nondecreasing such that 
\begin{equation}\label{propchi}
\chi(\rho)=\rho\text{ if }|\rho-1|<1/3,\
\chi|_{]-\infty,1/2]}=1/2,\ \chi_{[2,\infty[}=2 ,
\end{equation}
and define $\widetilde{G}$ as follows: since $G''(1)=g'(1)=1,\ G'(1)=g(1)=0$, we have 
$G(\rho)\geq (\rho-1)^2/3$ on some interval $(1-\delta,1+\delta)$, according to Borel's lemma 
there exists a smooth extension $\widetilde{g}$ of $g$ on $[1+\delta,1+2\delta]$, $[1-2\delta,1-\delta]$ such 
that for any $k\in \mathbb{N}$, 
$\widetilde{g}^{(k)}(1\pm 2\delta)=\frac{d^k(\rho-1)}{d\rho^k}(1\pm2\delta)$, and 
$\widetilde{G}>0$, then we set $\widetilde{g}=\rho-1$ on $(1-2\delta,1+2\delta)^c$, $\widetilde{G}
=\int_1^{\rho}\widetilde{g}(r)dr$. The function $\widetilde{G}$ satisfies
\begin{equation}\label{propgtilde}
\widetilde{G}\in C^\infty (\R),\ \widetilde{G}|_{(1-\delta,1+\delta)}=G,\ 
\widetilde{G} \gtrsim (\rho-1)^2,\ |\widetilde{G}'|\lesssim |\rho-1|.
\end{equation}
Now let us set 
\begin{equation}\label{modenergy}
\widetilde{E}(\rho,\phi)=\int_{\R^d}\frac{1}{2}(\chi(\rho)|\nabla \phi|^2+K(\chi(\rho))
|\nabla \rho|^2)+\widetilde{G}(\rho)dx.
\end{equation}
Obviously, if $\|\rho-1\|_\infty$ is small enough then $E(\rho,\phi)=\widetilde{E}(\rho,\phi)$, and from 
\eqref{propchi},\eqref{propgtilde}
\begin{equation*}
\forall\,(\rho,\phi)\in \mathcal{H},\ \widetilde{E}(\rho,\phi)\gtrsim \|\rho-1\|_{H^1}^2+\|\nabla 
\phi\|_2^2.
\end{equation*}
If $(\rho,\phi)$ is a solution of the minimization problem  
\begin{equation}
\label{minpb} 
\inf \{ \widetilde{E}(\rho,\phi),\ 
(\rho,\phi)\in \mathcal{H}:\ P(\rho,\phi)=p\},
\end{equation}
it should satisfy the following Euler-Lagrange equations where $\widetilde{g}:=\widetilde{G}'$ 
\begin{equation}\label{TWmod}
\exists\, c:
\left\{
\begin{array}{ll}
\displaystyle 
-c\partial_1\rho +\text{div}(\chi(\rho)\nabla \phi)=0,\\
\displaystyle 
-c\partial_1\phi+\chi'(\rho)\frac{|\nabla \phi|^2}{2}-K(\chi(\rho))\Delta\rho-\frac{(K\circ\chi)'
|\nabla\rho|^2}{2}+\widetilde{g}(\rho)=0.
\end{array}
\right.
\end{equation}
Since a solution of \eqref{TWmod} such that $\|\rho-1\|_\infty<<1$ is a solution of 
\eqref{TWEK}, our approach  will be to prove that for $p$ small enough, there exists 
existence of a solution to \eqref{minpb}, the minimizer is smooth and satisfies $\|\rho-1\|_\infty<<1$.  
We can now give a precise statement of our result, that we chose to state for the non-rescaled variables 
in order to underline the role of the physical variables.
\begin{theo}\label{maintheo}
Let $\rho_0\in \R^{+*}$ such that $g'(\rho_0)>0$, for $p>0$ we set
\begin{equation}
\widetilde{E}_{\min}(p):=\inf_{(\rho,\phi)\in (\rho_0+H^1)\times \dot{H}^1, P(\rho,\phi)=p} 
\widetilde{E}(\rho,\phi).
\end{equation}
Under the assumption $\Gamma:=3+\frac{\rho_0g''(\rho_0)}{g'(\rho_0)}\neq 0$, there exists
$p_0>0$ such that for any $0\leq p\leq p_0$, the infimum is attained at a minimizer 
$(\rho_p,\phi_p)\in \cap_{j\geq 0}(\rho_0+H^j)\times \dot{H}^j$, such that $(\rho_p,\phi_p)$ is a 
solution of \eqref{TWEKnonrescal} for some 
$c_p>0$. Moreover let $c_s=\sqrt{\rho_0g'(\rho_0)}$, then 
\begin{eqnarray}
\exists\,\alpha,\beta>0:\,\forall\,0\leq p\leq p_0,\ 
\ c_s p-\beta p^3&\leq&  \widetilde{E}_{\min}(p)=E(\rho_p,\phi_p)\leq c_sp-\alpha p^3,\\
\label{csharp} c_s-\beta p^2&\leq& c_p\leq c_s-\alpha p^2.
\end{eqnarray}
\end{theo}
\begin{rmq}
It is not clear if $(\rho_p,\phi_p)$ is a constrained minimizer of $E$.
\end{rmq}
\begin{rmq}
The assumption $\Gamma\neq 0$ is not technical. In the case of NLS it appears in the recent paper \cite{ChirMar} as 
necessary and sufficient for the strict concavity of $E_{NLS,\min}$ near $0$, a condition which is important, if not 
unavoidable, for the construction of minimizers. If $\Gamma=0$ scattering of solutions of \eqref{EK} 
for small data seems expectable but remains so far open, even for NLS.
\end{rmq}

\subparagraph{Idea of proof}
We first point out that (contrary to the NLS case), it is not easy to get elliptic regularity 
from equations \eqref{TWEK}, indeed they basically look like $\Delta f=|\nabla f|^2$, and the argument 
$f\in H^1\Rightarrow 
|\nabla f|^2\in L^1\Rightarrow f\in W^{2,1}\hookrightarrow H^1$ does not allow to bootstrap trivially regularity. 
But since the failure is somewhat ``critical'', working with $\|(\rho,\phi)\|_{\mathcal{H}}<<1$ allows to overcome 
this issue (this is done in proposition \ref{mainest}).\vspace{1mm}\\
The major issue is of course the defect of compactness in $\R^2$. 
In the spirit of \cite{BGS}, this is overcome by first solving the minimization 
problem on large tori $\mathbb{T}_n^2=(\R/2\pi n\Z)^2$, on which thanks to the compact embedding 
$H^1\hookrightarrow L^2$ the existence of a minimizer is easy. 
In order to handle smoothness 
issues, we follow a regularization procedure: we use a ``mollified 
energy'' $$\widetilde{E}_\varepsilon^n(\rho,\phi)=\widetilde{E}+\frac{\varepsilon}{2}
\int_{\mathbb{T}_n^2} |\Delta \phi|^2+|\Delta \rho|^2dx$$ and prove that minimizers of small energy 
satisfy elliptic estimates independent of $\varepsilon$ (smallness is essential for this step). 
Letting $\varepsilon\rightarrow 0$, this provides a solution $(\rho_n,\phi_n)$ 
of the minimization problem on $\mathbb{T}_n^2$. Next we let $n\rightarrow \infty$, and prove  
the convergence of $(\rho_n,\phi_n)$ to $(\rho,\phi)$, solution of the constrained 
minimization problem \eqref{minpb}.\\
The main tool to get some compactness of the sequence $(\rho_n,\phi_n)$ is the 
strict concavity\footnote{Actually the key is not concavity, but a consequence: strict subadditivity.} 
of the minimal energy $\widetilde{E}_{\min}(p)$
which is obtained by mixing general abstract arguments and \emph{ad hoc} computations. To get a feeling of 
how concavity is used, 
consider the following simplified version of dichotomy in Lions's concentration compactness 
principle: assume that instead of converging to a minimizer, $(\rho_n,\phi_n)$ splits in two 
parts, namely there exists functions $(\rho^1,\phi^1)$, $(\rho^2,\phi^2)$ such that 
$\widetilde{E}_{\min}^n(\rho_n,\phi_n)\longrightarrow \widetilde{E}(\rho^1,\phi^1)+
\widetilde{E}(\rho^2,\phi^2)$, $P(\rho_n,\phi_n)\longrightarrow P(\rho^1,\phi^1)
+P(\rho^2,\phi^2)$. Then passing to the limit in $n$ we have
\begin{eqnarray*}
P(\rho^1,\phi^1)+P(\rho^2,\phi^2)&=&p_1+p_2=p,\\
\widetilde{E}_{\min}(p_1+p_2)&=&\widetilde{E}(\rho^1,\phi^1)+\widetilde{E}(\rho^2,\phi^2)
\geq \widetilde{E}_{\min}(p_1)+\widetilde{E}_{\min}(p_2).
\end{eqnarray*}
On the other hand, by strict subadditivity $\widetilde{E}_{\min}(p_1+p_2)<\widetilde{E}_{\min}(p_1)
+\widetilde{E}_{\min}(p_2)$, which is a contradiction. For a remarkably clear and general discussion 
on this strategy, we refer to the seminal paper of P.L. Lions \cite{LionsCC}.
% \\
% Even in the case of the Gross-Pitaevskii equation, it is interesting to work 
% directly on the fluid variables $(\rho,\phi)$ as it allows to avoid difficulties such as the 
% definition of the momentum or the topological degree of $\psi$. Of course the main drawback is 
% that this approach only works for $\rho$ bounded away from zero, a feature known to be untrue for 
% the Gross-Pitaevskii equation if $p$ is large, see \cite{BOS}. This limitation seems to highlight 
% the difficulty of constructing solutions of \eqref{EK} with vortices.
\subparagraph{Plan of the paper}
The rest of the article is organized as follows: in section $2$, we prove a key elliptic 
estimate for solutions of \eqref{TWmod} in a simple case, in section $3$ we establish 
the concavity of $\widetilde{E}_{\min}$ and the upper bound $\wE{\min}\leq p-\alpha p^3$ from 
which we deduce its strict subadditivity.
Theses sections are preliminaries to section $4$ where we prove 
the existence of solutions to the minimization problem \eqref{minpb}: we first study the 
minimization problem on $\mathbb{T}_n^2$ for fixed $n$. We obtain the existence of constrained 
minimizers for the mollified energy $\widetilde{E}_\varepsilon^n$, 
from which we deduce the existence of smooth minimizers for the nonregularized problem. Letting 
$n\rightarrow \infty$, we obtain the convergence of minimizers on $\mathbb{T}_n^2$ to a minimizer 
on $\R^2$ with a concentration compactness argument. 
Finally we complete the a priori estimates of $\widetilde{E}_{\min}$ and $c$ thanks to Pohozaev type 
identities in section \ref{poho}.
The concentration compactness argument relies on a kind of profile decomposition 
essentially similar to the one in \cite{BGS}, for completeness we prove its existence in the 
appendix \ref{appendiceprof}.
In appendix \ref{unde} we discuss the one-dimensional case, where explicit computations allow 
to observe very strong similarities with one-dimensional NLS, and prove a new nonlinear instability 
property of some solitary waves.

\subparagraph{Notations} If $a\leq Cb$, $a,b,C>0$, $C$ a constant independent of the parameters, 
we write  $a\lesssim b$. If $C_1a\leq b\leq C_2a$ with $C_1,C_2$ positive constants, we write 
$a\sim b$ if there is no ambiguity with the usual meaning of $\sim$. We denote the Fourier 
transform of an application $\phi$ as $\widehat{\phi}$.\\
As mentioned in the introduction, $\mathcal{H}=\{(\rho,\phi)\in (1+H^1)\times \dot{H}^1\}$.

\section{An elliptic estimate}\label{secell}
We first clarify our functional framework.
\begin{define}\label{defH1}
The space $\dot{H}^1$ is the set of $\phi\in L^2_{\text{loc}}$ such that 
$\nabla \phi\in L^2$ in the distributional sense. We define $L^2_{\text{curl}}:=\nabla \dot{H}^1$
with norm $\|\nabla\phi\|_{L^2}$.
\end{define}
We shall need the following standard density result.
\begin{prop}\label{anfunc} 
$L^2_{\text{curl}}$ coincides with $\{u\in L^2:\ \text{curl}(u)=0\}$, and thus 
is a Hilbert space, in which $\nabla C^\infty_c(\R^2)$ is dense.
\end{prop}
\begin{proof}
For the first part see e.g. \cite{Simon} or \cite{TartarTopics}. For the second part it suffices 
to check $(\nabla C_c^\infty)^{\perp}=\{0\}$. If $\nabla \phi\in (\nabla C_c^\infty)^{\perp}$ then 
$\Delta \phi=0 \ (\mathcal{D}')$, and $\nabla \phi\in L^2$ thus $\Delta\phi=0\ (\mathcal{S}')$. 
Thus $|\xi|^2\widehat{\varphi}=0$. This implies 
that $\widehat{\varphi}$ is a linear combination the Dirac distribution at $0$ and its first order 
derivatives, equivalently $\varphi$ is a first order polynomial. The condition 
$\nabla \phi\in L^2$ then implies that $\varphi$ is a constant, so that $u=0$. 
\end{proof}

\begin{prop}\label{mainest}
Let $M>0$, $(\rho,\phi)$ a solution of \eqref{TWmod} with $|c|\leq M$. There exists 
$\varepsilon,C>0$ depending only on $M$\footnote{of course it depends also on 
$K\circ \chi$ and $\widetilde{G}$ but it does not matter for the analysis.}
such that 
\begin{equation*}
(\rho,\nabla\phi)\in 1+H^2\times H^1 \text{ and } \widetilde{E}(\rho,\phi)<\varepsilon
\Rightarrow \|\rho-1\|_\infty<C\sqrt{\widetilde{E}(\rho,\phi)}.
\end{equation*}
In particular for $\widetilde{E}(\rho,\phi)$ small enough, a (smooth) solution of \eqref{TWmod} is a 
traveling wave of the Euler-Korteweg system.
\end{prop}

\begin{proof}
Setting $u=\nabla \phi$ and denoting $\chi$ for $\chi(\rho)$ we have from the first equation
\begin{equation*}
\Delta \phi+\frac{\nabla\chi\cdot u}{\chi}-c\frac{\partial_1\rho}{\chi}=0
\Rightarrow 
\Delta u+\nabla \big(\nabla \ln(\chi)\cdot u\big)-c\nabla \bigg(\frac{\partial_1\rho}{\chi}\bigg)
=0
\end{equation*}
Taking the gradient of the equation, the scalar product with $u$ and integrating, we get 
\begin{equation*}
\int |\nabla u|^2+\bigg(\nabla \ln\chi \cdot u-c\frac{\partial_1\rho}{\chi}\bigg)\text{div}u\,dx=0
\end{equation*}
so that from Cauchy-Schwarz's inequality
\begin{equation*}
\frac{1}{2}\|\nabla u\|_2^2\leq 2\|\nabla \ln \chi\cdot u\|_2^2+2c\|\partial_1\rho/\chi\|_2^2\leq  
2\|\nabla \ln \chi\|_4^2\|u\|_4^2+2c\|\dun\rho/\chi\|_2^2.
\end{equation*}
If $d=2$, we use $\|u\|_4\lesssim \|u\|_{\dot{H}^{1/2}}\lesssim \|u\|_2^{1/2}
\|\nabla u\|_{2}^{1/2}$ so that 
\begin{eqnarray*}
\|u\|_4^2&\leq& C\|u\|_2\|\nabla u\|_2\leq \sqrt{2}C\|u\|_2(\|\nabla \ln \chi\|_4\|u\|_4+
+c\|\dun\rho/\chi\|_2),\\
\Rightarrow \|u\|_4^2&\leq& \frac{C}{\sqrt{2}}\|u\|_2^2\|\nabla \ln \chi\|_4^2
+MC\sqrt{2}\|u\|_2\|\dun\rho/\chi\|_2\\
&\lesssim & \|u\|_2^2\|\nabla \rho\|_4^2+\|u\|_2\|\dun\rho\|_2,
\end{eqnarray*}
where we used $\|\chi'\nabla \rho\|_2\lesssim \|\nabla \rho\|_2$. 
Next we rewrite the momentum equation as 
\begin{equation*}
\Delta \rho =\frac{-c\dun\rho}{K\circ \chi}+\frac{\chi'}{2K\circ \chi}|\nabla \phi|^2
-\frac{(K\circ \chi)'}{2K\circ \chi}|\nabla \rho|^2+\frac{\widetilde{g}}{K\circ \chi}.
\end{equation*}
Since $K$ is smooth, positive on $]0,\infty[$, $(K\circ\chi)'$ and $1/K\circ \chi$ are uniformly 
bounded, and from \eqref{propgtilde} taking the $L^2$ norm gives
\begin{eqnarray*}
\|\Delta \rho\|_2&\lesssim& \|\partial_1\rho\|_2+\|\nabla \phi\|_4^2+\|\nabla \rho\|_4^2
+\|\rho-1\|_{2}\\
&\lesssim& \|\partial_1\rho\|_2+\|\partial_1\rho\|_2\|u\|_2+\|\nabla \rho\|_4^2
+\|u\|_2^2\|\nabla \rho\|_4^2+\|\rho-1\|_2.
\end{eqnarray*}
Next we use again Sobolev's embedding 
$\|\nabla \rho\|_4^2\lesssim \|\nabla\rho\|_2\|\Delta\rho\|_2$ which gives
\begin{equation*}
\|\Delta \rho\|_2\leq C\big(\|\partial_1\rho\|_2+\|\partial_1\rho\|_2\|u\|_2
+\|\rho-1\|_2+(1+\|u\|_2^2)\|\nabla \rho\|_2\|\Delta\rho\|_2\bigg).
\end{equation*}
We recall that 
$\widetilde{E}(\rho,\phi)\gtrsim \|\rho-1\|_{H^1}^2+\|\nabla \phi\|_2^2$, so that if 
$\widetilde{E}(\rho,\phi)$is small enough, $C(1+\|u\|_2^2)\|\nabla \rho\|_2<1/2$ and we deduce 
\begin{equation*}
\|\Delta \rho\|_2\leq C\big(\|\partial_1\rho\|_2+\|\partial_1\rho\|_2\|u\|_2
+\|\rho-1\|_2)\lesssim \sqrt{\widetilde{E}(\rho,\phi)}.
\end{equation*}
From Sobolev's embedding we conclude $\|\rho-1\|_{\infty}\lesssim \|\rho-1\|_{H^2}\lesssim 
\sqrt{\widetilde{E}(\rho,\phi)}$. In particular if the energy is small enough 
$\chi(\rho)=\rho,\ \widetilde{G}(\rho)=G(\rho)$ and $\rho$ is a solution of \eqref{TWEK}.
\end{proof}

\section{Properties of the energy}
We recall $\widetilde{E}_{\text{min}}(p)=\inf_{P(\rho,\phi)=p}\widetilde{E}(\rho,\phi)$. 
We start with some properties that are true in generic minimization settings (continuity, 
concavity of $\wE_{\min}$)
before tackling the strict subadditivity of $\wE_{\min}$, where we use the structure of 
$\wE$ and $P$.
\begin{lemma}\label{smoothmin} For any $p\geq 0$, there exists a minimising sequence 
$(\rho_n,\nabla \phi_n) \in (1+C_c^\infty(\R^2))\times C^\infty_c(\R^2)$.
\end{lemma}
\begin{proof}
The case $p=0$ is obvious. For $p>0$ it suffices to prove that for any 
$(\rho-1,\nabla\phi)\in H^1\times L^2$, there exists $(\rho_n-1,\phi_n)\in (C_c^\infty(\R^2))^2$ 
such that $P(\rho_n,\phi_n)=p,\ \widetilde{E}(\rho_n,\phi_n)\rightarrow \widetilde{E}(\rho,\phi)$. 
By density (prop. \ref{anfunc}), there exists $(r_n-1, \nabla\psi_n)\in C_c^\infty (\R^2)$ 
such that 
\begin{equation*} 
\|r_{n}-\rho\|_{H^1}+\|\nabla\psi_{n}-\nabla \phi\|_{L^2}
\rightarrow_n 0,\ r_n\rightarrow 
\rho\ a.e.\ .
\end{equation*}
Clearly $P(r_n,\psi_n)\rightarrow p$, $\int_{\mathbb{\R}^2}\widetilde{G}(r_n)\rightarrow 
\int_{\R^2} \widetilde{G}(\rho)dx$, and up to an extraction such that $r_n\rightarrow \rho\ a.e.$ 
we have by dominated convergence
\begin{eqnarray*}
\int_{\R^2} \chi(r_n)|\nabla \psi_n|^2-\chi(\rho)|\nabla \phi|^2dx&=&
\int_{\R^2} (\chi(r_n)-\chi(\rho))|\nabla \phi|^2+\chi(r_n)
(|\nabla \phi|^2-
|\nabla \psi_n|^2)dx,\\
&\longrightarrow_n& 0,
\end{eqnarray*}
\begin{equation*}
\int_{R^2}K(\chi(r_n))|\nabla r_n|^2-K(\chi(\rho))|\nabla \rho|^2dx
\longrightarrow_n 0,
\end{equation*}
from which we deduce $E(r_n,\psi_n)-E(\rho,\phi)\rightarrow 0$.
Let $\varepsilon_n=p-P(r_n,\psi_n)$, we construct a slight modification 
$(\rho_n,\phi_n)$ of $(r_n,\psi_n)$ such that
$P(\rho_n,\phi_n)=p$\ :
let $\varphi\in C_c^\infty(\R^2)$, $A:=\partial_1\varphi$ with $\|\partial_1\varphi\|_2=1$.
Up to a translation (that depends on $n$), 
we can assume $\text{supp}(\varphi)\cap (\text{supp}(1-r_n)\cup \text{supp}
(\psi_n))=\emptyset$, and define
\begin{equation*}
\rho_n=r_n+\text{sign}(\varepsilon_n)\sqrt{|\varepsilon_n|}A,\ 
\phi_n=\psi_n+\sqrt{|\varepsilon_n|}\varphi.
\end{equation*}
We conclude
\begin{eqnarray*}
P(\rho_n,\phi_n)&=&P(r_n,\psi_n)+\varepsilon_n=p,\\ 
\widetilde{E}(\rho_n,\phi_n)&=&\widetilde{E}(r_n,\psi_n)+\frac{1}{2}\varepsilon_n
\int_{\R^2}K(\chi(\rho_n)|\nabla A|^2\\
&&\hspace{4cm}+\chi(\rho_n)|\nabla \varphi|^2+O(A^2)
\,dx\\
&=& \widetilde{E}(r_n,\psi_n)+O(\varepsilon_n)\longrightarrow_n \widetilde{E}(\rho,\phi).
\end{eqnarray*}
\end{proof}

\begin{prop}\label{concE}
The application $p\in \R^+\mapsto 
\widetilde{E}_{\text{min}}(p)$ is $1$-Lipschitz, concave, non decreasing.
\end{prop}
\begin{proof}We split the proof in three steps:
\paragraph{$\widetilde{E}_{\text{min}}$ is Lipschitz}
Let $p<q$, $\delta>0$ to be fixed, according to lemma $\ref{smoothmin}$, there exists 
$(\rho-1,\phi)\in (C_c^\infty)^2$ such that $P(\rho,\phi)=p$, $\widetilde{E}(\rho,\phi)\leq 
\widetilde{E}_{\text{min}}(p)+\delta$. Combining proposition $\ref{estimenergie}$ and lemma 
$\ref{smoothmin}$ there exists 
$(\rho_0-1,\phi_0)\in (C_c^\infty(\R^2))^2$ such that $P(\rho_0,\phi_0)=q-p$, 
$\widetilde{E}(\rho_0,\phi_0)\leq q-p$. Up to a translation, we can assume $(\rho_0-1,\phi_0)$ 
have disjoint support with $(\rho-1,\phi)$, so that 
\begin{equation*}
P(\rho+\rho_0-1,\phi+\phi_0)=p+q-p=q,\
 \widetilde{E}(\rho+\rho_0-1,\phi+\phi_0)\leq \widetilde{E}_{\text{min}}(p)+\delta+q-p.
\end{equation*}
Since $\delta$ is arbitrary, $\widetilde{E}_{\text{min}}(q)-\widetilde{E}_{\text{min}}(p)\leq q-p$.
The reverse inequality can be obtained with a similar argument (using $-\phi_0$ instead of 
$\phi_0$).
\paragraph{$\widetilde{E}_{\text{min}}$ is concave} Since $\widetilde{E}_{\text{min}}$ is 
continuous, it suffices to prove 
that for any $p_1<p_2\in [0,p_0]$, $\widetilde{E}_{\text{min}}((p_1+p_2)/2)\geq
\frac{\widetilde{E}_{\text{min}}(p_1)+\widetilde{E}_{\text{min}}(p_2)}{2}$. This relies on a classical reflection argument. 
For $f$ defined on $\R^2$, we define $T_a^+(f)$ (resp $T_a^-f$) as the function symmetric with 
respect to the line $x_2=a$ and that coincides with $f$ on $x_2>a$ (resp. $x_2<a$). The maps 
$T_a^+,T_a^-$ are linear continuous $H^1\rightarrow H^1$, and from Lebesgue's dominated 
convergence theorem $T_a^+\rightarrow_{+\infty} 0,\ T_a^-\rightarrow_{-\infty}0$, $a\mapsto 
T_a^\pm$ is continuous. This also implies
$$
\|T_a^+f\|_{L^2}\rightarrow_{-\infty}2\|f\|_{L^2},\ 
\|\nabla T_a^+f\|_{L^2}\rightarrow_{-\infty}2\|\nabla f\|_{L^2},
$$
and the symmetric property for $T_a^-$. We also note that for any function $F$, as soon as the 
integrals make sense
\begin{equation}\label{sympratique}
\int_{\R^2}F(T_a^+f)+F(T_a^-f)dx=2\int_{\R^2}F(f)dx. 
\end{equation}
Now let $\delta>0$, $(\rho,\phi)$ be such that 
$P(\rho,\phi)=\frac{p_1+p_2}{2},\ \widetilde{E}(\rho,\phi)\leq \widetilde{E}_{\text{min}}\bigg(\frac{p_1+p_2}{2}\bigg)
+\delta$. Since 
$\displaystyle \lim_{+\infty}\|P(T_a^+(\rho,\phi))\|_2=0<p_1$, there exists $a_1$ such that 
$P(T_{a_1}^+(\rho,\phi))=p_1$, and from \eqref{sympratique}, $P(T_{a_1}^-(\rho,\phi))=p_2$. Then 
using again \eqref{sympratique}
\begin{equation*}
\widetilde{E}_{\text{min}}(p_1)+\widetilde{E}_{\text{min}}(p_2)\leq \widetilde{E}\big(T_{a_1}^+(\rho,\phi)\big)
+\widetilde{E}\big(T_{a_1}^-(\rho,\phi)\big)\leq 2\widetilde{E}_{\text{min}}\bigg(\frac{p_1+p_2}{2}\bigg)+2\delta.
\end{equation*}
Since $\delta$ is arbitrary, we get $\widetilde{E}_{\text{min}}((p_1+p_2)/2)\geq \frac{\widetilde{E}_{\text{min}}(p_1)+
\widetilde{E}_{\text{min}}(p_2)}{2}$.
\paragraph{$\wE_{\text{min}}(p)$ is non decreasing} Obvious since it is concave and nonnegative.
\end{proof}

The next proposition gives a sharp upper bound for $\widetilde{E}_{\min}$.
\begin{prop}\label{estimenergie}
There exists $p_0>0,\ \alpha>0$ such that
\begin{equation*}
\forall\,0<p<p_0,\ \exists\,(\rho_p,\phi_p)\in \mathcal{H}:\ P(\rho_p,\phi_p)=p,\ 
\widetilde{E}(\rho_p,\phi_p)\leq p-\alpha p^3.
\end{equation*}
In particular $\widetilde{E}_{\text{min}}(p)\leq p-\alpha p^3$. Moreover, up to taking a smaller 
$p_0$ if $(\rho,\phi)$ is a minimiser, then $\|\rho-1\|_\infty \gtrsim p^2$.
\end{prop}
\begin{proof}
The idea is to construct an approximate minimizer by using the following formal asymptotic 
(rigorously justified for the Gross-Pitaevskii 
equation \cite{BGS2}): set $\rho=1+\varepsilon^2A_\varepsilon(z_1,
z_2),\ \phi=\varepsilon\varphi_\varepsilon(z_1,z_2)$, $z_1=\varepsilon x_1,\ z_2=\varepsilon^2x_2$.
If $(\rho,\phi)$ is a solution of \eqref{TWEK} with speed $c=\sqrt{1-\varepsilon^2}$, the mass 
conservation reads
\begin{eqnarray*}
-c\partial_1A_\varepsilon+\partial_1^2\varphi_\varepsilon
+\varepsilon^2(\partial_2^2\varphi_\varepsilon+A_\varepsilon\partial_1^2\varphi_\varepsilon
+\partial_1A_\varepsilon\partial_1\varphi_\varepsilon)=O(\varepsilon^4).
\end{eqnarray*}
Next using Taylor's expansion $\widetilde{g}=\varepsilon^2A_{\varepsilon}+g''(1)\varepsilon^4
\frac{A_\varepsilon^2}{2}+O(\varepsilon^4)$, the momentum equation gives 
\begin{eqnarray*}
-c\partial_1\varphi_\varepsilon+A_\varepsilon+\varepsilon^2\bigg(\frac{g''(1)A_\varepsilon^2+
(\partial_1\varphi_\varepsilon)^2}{2}
-K(1)\partial_1^2A_\varepsilon\bigg)=O(\varepsilon^4).
\end{eqnarray*}
At first order, we have $\dun\phi_\varepsilon=A_\varepsilon+O(\varepsilon^2)$, next if we 
multiply the mass equation by $c$, apply $\dun$ to the momentum equation and add them, we get
\begin{eqnarray}
\label{kpasymp}
\dun A_\varepsilon+\partial_2^2\dun^{-1}A_\varepsilon+(3+g''(1))A_\varepsilon\dun A_\varepsilon
-K(1)\dun^3A_\varepsilon=O(\varepsilon^2),
\end{eqnarray}
Note that $\gamma:=3+g''(1)$ is the rescaled version of $\Gamma=3+\rho_0g''(\rho_0)/g'(\rho_0)$ 
thus by assumption $\gamma\neq 0$.
\eqref{kpasymp} is a KP1 type equation, the normalized KP1 equation is 
\begin{equation}
\label{kpadim}
\partial_1w+w\dun w+\partial_2^2\dun^{-1}w-\dun^3w=0
\end{equation}
One can pass from a solution of \eqref{kpadim} to a solution of \eqref{kpasymp} by setting 
\begin{equation}\label{corresp}
A=\frac{1}{\gamma}w(x_1/\sqrt{K(1)},x_2/\sqrt{K(1)}).
\end{equation}
In \cite{BouardSaut}, solutions of the KP equation are constructed, any such solution satisfy
\begin{equation}\label{energie}
E_{KP}(w):=
\frac{1}{2}\int_{\R^2} |\partial_2\dun^{-1}w|^2+w^3/3+|\dun w|^2dx<\infty,
\end{equation}
Moreover, such solutions are smooth, belong 
to $L^q$ for any $1<q\leq \infty$ as well as their gradients, there exists a smooth $v\in L^p$ 
for any $2<p\leq \infty$ such that $\dun v=w$, $\nabla v\in L^q$, and 
$E_{KP}(w)=-\int_{\R^2} w^2/6<0$ (see \cite{BouardSaut} or \cite{BGS} p.41). Let $w$ be such a solution
\footnote{some optimization can be done by choosing a so-called ground state, but it is not really useful here.}, 
we define $A$ by \eqref{corresp}, and set $\rho=1+\varepsilon^2A(\varepsilon x_1,\varepsilon^2x_2)$, 
$\phi=\varepsilon\dun^{-1} A(\varepsilon x_1,\varepsilon^2x_2)$. Since $A$ is bounded, 
$|\rho-1|=O(\varepsilon^2)$ so that $E$ and $\widetilde{E}$ coincide for $\varepsilon$ small 
enough. We have $E_{\widetilde{KP}}(A)=\frac{K(1)}{\gamma^2}E_{KP}(w)<0$, and basic computations give
\begin{equation*}
P(\rho,\phi)=\varepsilon^4\int A^2(\varepsilon x_1,\varepsilon^2x_2)dx=\varepsilon\|A\|_2^2,
\end{equation*}
\begin{eqnarray*}
\widetilde{E}(\rho,\phi)&=&\frac{1}{2}\int_{\R^2}(1+\varepsilon^2A)\big(\varepsilon^4A^2
+\varepsilon^6(\partial_2\dun^{-1}A)^2\big)+K(1+\varepsilon^2A)
\big(\varepsilon^6(\dun A)^2+\varepsilon^8(\partial_2A)^2\big)\\
&&+\varepsilon^4A^2(\varepsilon x_1,\varepsilon^2x_2)+(2G(1+\varepsilon^2A)-\varepsilon^4A^2)dx\\
&\leq &\varepsilon\|A\|_2^2+\frac{\varepsilon^3}{2}\int_{\R^2}(\partial_2\dun^{-1}A)^2+A^3
+K(1)(\dun A)^2(z_1,z_2)+\frac{(\max \widetilde{G}''')A^3}{3} dz+R
\end{eqnarray*}
where $\displaystyle R=\frac{\varepsilon^5}{2}\int A(\partial_2\dun^{-1}A)^2
+\frac{K(1+\varepsilon^2A)-1}{\varepsilon^2}(\dun A)^2+K(1+\varepsilon^2A)(\partial_2A)^2dx
=O(\varepsilon^5)$. As a consequence by definition of $\widetilde{E}_{\text{min}}(p)$
\begin{equation*}
\widetilde{E}_{\text{min}}(\varepsilon\|A\|_2^2)\leq \widetilde{E}(\rho_\varepsilon,\phi_\varepsilon)
\leq \varepsilon\|A\|_2^2+\varepsilon^3E_{\widetilde{KP}}(A)+C\varepsilon^5,
\end{equation*}
taking $\varepsilon\leq \sqrt{-E_{\widetilde{KP}}/(2C)}$ completes the first part of the proof.\\
Now if $(\rho,\phi)$ is a minimiser, from $\int_{\R^2} \widetilde{G}(\rho)=
\big(1+O(\|\rho-1)\|_\infty)\big)\int_{\R^2} \frac{(\rho-1)^2}{2}dx$ we have
\begin{eqnarray*}
p=\int_{\R^2}(\rho-1)\dun \phi\leq \frac{1}{\displaystyle \inf_{\R^2}\sqrt{\chi(\rho)}}
\int_{\R^2} \frac{(\rho-1)^2
+\chi|\nabla \phi|^2}{2}dx
&\leq& \frac{\big(1+O(\|\rho-1)\|_\infty)\big)\widetilde{E}_{\min}(p)}
{\inf \sqrt{\chi}}\\
&\leq & \frac{\big(1+O(\|\rho-1)\|_\infty)\big)(p-\alpha p^3)}
{\inf \sqrt{\chi}}.
\end{eqnarray*}
There are two possibilities: 
\begin{itemize}
 \item if $\inf \sqrt{\chi}\leq 1-\alpha p^2/2$, then $\inf \sqrt{\rho}
\leq 1-\alpha p^2$
\item else $1+O(\|\rho-1\|_\infty)\geq \inf \sqrt{\chi}/(1-\alpha p^2)\geq 
1+\alpha p^2/2+O(p^4)$, then $\|\rho-1\|_\infty\gtrsim p^2$.
\end{itemize}
\end{proof}

As pointed out in the introduction, rather than concavity we will use subadditivity:
\begin{prop}\label{posD}
The application $\wE_{\min}:\R^+\rightarrow \R^+$ satisfies the following properties: 
\begin{enumerate}
\item it is differentiable at $p=0$ and $\wE_{\min}'(0)=1$.
\item it is strictly subadditive : $\forall\,0<p_1,p_2,\ \wE_{\min}(p_1+p_2)<\wE_{\min}(p_1)
+\wE_{\min}(p_2)$.
\item the application 
$$(p_1,p_2)\in (\R^+)^2\rightarrow D(p_1,p_2):=\widetilde{E}_{\min}(p_1)
+\widetilde{E}_{\min}(p_2)-\widetilde{E}_{\min}(p_1+p_2)$$ 
is nonnegative and nondecreasing in both $p_1$ and $p_2$. Moreover 
\begin{equation}
p_1,\ p_2>0\Rightarrow D(p_1,p_2)>0.
\end{equation}
\end{enumerate}
\end{prop}
\begin{proof}
\textbf{1.} From proposition \ref{estimenergie} we have 
$\overline{\lim}_0\frac{\wE_{\min}(p)}{p}\leq 1$. Conversely, consider a sequence $p_n\rightarrow 0$, 
and pick approximate minimizers $(\rho_n,\phi_n)$ such that 
$$\forall\,n\geq 1,\ P(\rho_n,\phi_n)=p_n,\ \wE(\rho_n,\phi_n)\leq \wE_{\min}(p_n)(1+1/n).$$
Then $\|\rho_n-1\|_{H^1}^2+\|\nabla \phi_n\|_{L^2}^2\sim \wE(\rho_n,\phi)\longrightarrow_n 0$ 
and from Young's inequality
\begin{eqnarray*}
p_n=P(\rho_n,\phi_n)\leq \int_{\R^2}\frac{(\rho_n-1)^2}{2\chi(\rho_n)}
+\frac{\chi|\nabla \phi_n|^2}{2}dx
\end{eqnarray*}
Since $\widetilde{G}''(1)=1,\ \widetilde{G}'(1)=\widetilde{G}(1)=0$ and 
$\widetilde{G}=O(\rho-1)^2$, from Taylor's expansion we have $\widetilde{G}(\rho)
-\frac{(\rho-1)^2}{2\chi}=O(\rho-1)^3$. Combining it with Sobolev's embedding $H^1\hookrightarrow L^3$
\begin{eqnarray*}
p_n\leq \int_{\R^2}\widetilde{G}(\rho_n)
+\frac{\chi|\nabla \phi_n|^2}{2}+O(\rho_n-1)^3dx&\leq& \widetilde{E}(\rho_n,\phi_n)+
O(\widetilde{E}(\rho_n,\phi_n))^{3/2}\\
&\leq& \wE_{\min}(p_n)(1+O(\wE_{\min}(p_n))^{1/2})(1+1/n)
\end{eqnarray*}
This readily implies $\underline{\lim}_n\frac{\wE_{\min}(p_n)}{p_n}\geq 1$, and 1) is thus true.
\vspace{1mm}\\
\textbf{2.} This is a basic concavity argument. First we remark that $\wE_{\min}$ is not 
linear on any interval $[0,p]$, since $\wE_{\min}'(0)=1$ and $\wE_{\min}(p)<p$.
Assume there exists $p_1\leq p_2$ such that 
$\widetilde{E}_{\min}(p_1+p_2)=\widetilde{E}_{\min}(p_1)+\widetilde{E}_{\min}(p_2)$. 
For a unified treatment, if $p_1=p_2$ we write 
$\displaystyle \frac{\wE_{\min}(p_2)-\wE_{\min}(p_1)}{p_2-p_1}$ for the right derivative of 
$\wE_{\min}$. By concavity and using $\widetilde{E}_{\min}(0)=0$
\begin{equation*}
\frac{\widetilde{E}_{\min}(p_1)}{p_1}\geq
\frac{\widetilde{E}_{\min}(p_2)-\widetilde{E}_{\min}(p_1)}{p_2-p_1}\geq 
\frac{\widetilde{E}_{\min}(p_1+p_2)-\widetilde{E}_{\min}(p_2)}{p_1+p_2-p_2}=
\frac{\widetilde{E}_{\min}(p_1)}{p_1}.
\end{equation*}
Therefore for any $p\in [p_1,p_1+p_2]$, $\frac{\widetilde{E}_{\min}(p)
-\widetilde{E}_{\min}(p_1)}{p-p_1}=\frac{\widetilde{E}_{\min}(p_1}{p_1}\Leftrightarrow 
\widetilde{E}_{\min}(p)=\frac{\widetilde{E}_{\min}(p_1)}{p_1}p.$ Also 
\begin{eqnarray*}
\forall\,p\in [0,p_1],\  \frac{\widetilde{E}_{\min}(p_1)}{p_1}
=\frac{\widetilde{E}_{\min}(p_1+p_2)-\widetilde{E}_{\min}(p_1)}{p_1+p_2-p_1}
\leq \frac{\widetilde{E}_{\min}(p_1)-\widetilde{E}_{\min}(p)}{p_1-p}\leq 
\frac{\widetilde{E}_{\min}(p_1)}{p_1} \\
\Rightarrow \frac{\widetilde{E}_{\min}(p_1)}{p_1}(p)=\frac{\widetilde{E}_{\min}(p_1)}{p_1}p,
\end{eqnarray*}
Hence, $\widetilde{E}_{\min}$ is linear on $[0,p_1+p_2]$, this is a contradiction.\vspace{1mm}\\
\textbf{3.} Direct consequence of the subadditivity, and the fact that for any concave 
function $f$, $x\rightarrow \frac{f(x+p_2)-f(x)}{p_2}$ is decreasing.
\end{proof}
\begin{rmq}
The better lower bound $\wE_{\min}(p)\geq p-\beta p^3$ is based on some Pohozaev's identities, that 
in turn require the existence of minimizers, therefore their proof is postponed to section 
\ref{poho}.
\end{rmq}

\section{Existence of minimizers}
The existence is obtained by following the procedure in \cite{BGS}, which is the following
\begin{itemize}
\item If one replaces $\R^2$ by the torus $\mathbb{T}^2_n=\R^2/(2n\pi\Z)^2$, the existence of a 
minimiser to \eqref{minpb} for any $p$ is easy thanks to elliptic estimates and compact embeddings.
\item Any minimiser $(\rho_n^p,u_n^p)$ satisfies $\|\rho_n^p-1\|\geq C p^2$, with $C$ independent 
of $n$ (torus version of proposition \ref{estimenergie}).
\item Letting $n\rightarrow \infty$, up to translation and extraction $(\rho_{n}^p,\phi_{n}^p)$ 
converges to
$(\rho^{\widetilde{p}},\phi^{\widetilde{p}})$, which is a \emph{non trivial} solution of equation 
\eqref{TWmod} with $\wE(\rho^p,\phi^p)\leq \wE_{\min}(p)$.
\item The sequence $(\rho_n^p,\phi_n^p)$ actually converges globally so that 
$P(\rho^p,\phi^p)=\lim_n P_n(\rho_n^p,\phi_n^p)=p$. This is the most difficult point, which 
requires a careful analysis of the difference between the energy density 
$K|\nabla \rho_n^p|^2+\chi(\rho_n^p)|\nabla\phi_n|^2+\widetilde{G}(\rho_n)$ and the momentum density 
$(\rho_n-1)\dun \phi_n$ on the ``vanishing set'' $|\rho_n-1|<<1$.
\end{itemize}
We point out that one of the reasons why F. Bethuel, P. Gravejat and J.C. Saut used the preliminary 
minimization on the torus was the difficulty to define $P(\rho,\phi)$. This is not an issue 
here, however working on the torus is essential to get strong a priori estimates and start a 
compactness procedure.\\
In order to use the (torus version of) elliptic estimate in proposition \ref{mainest}, we first 
use the smoothened energy 
\begin{equation*}
\widetilde{E}_n^\varepsilon(\rho,\phi):=\frac{1}{2}\int_{\mathbb{T}_n^2}\rho|\nabla \phi|^2+
K(\chi(\rho))|\nabla \rho|^2+2G(\rho)+\varepsilon\big((\Delta\phi)^2+(\Delta\rho)^2\big)dx,
\end{equation*}
and the notation $(\widetilde{E}_n^\varepsilon)_{\min}(p):=\inf_{P_n(\rho,\phi)=p}
\widetilde{E}_n^\varepsilon(\rho,\phi)$.
We provide a collection of lemmas that mimick the situation on $\R^2$, without the regularizing 
terms. 

\subsection{Minimizers on large tori}

The first step is a very rough version of proposition \ref{estimenergie}.

\begin{lemma}\label{estimeneps}
There exists $M>0$ such that for any $0\leq  p\leq 1$, $n\geq 5/p^2$, $0\leq\varepsilon\leq 1$, 
$$(\widetilde{E}_n^\varepsilon)_{\min}(p)\leq Mp.$$
\end{lemma}

\begin{proof}
We start with an ansatz similar to proposition \ref{estimenergie}: let $\theta\in C_c^\infty(\R^2)$
such that $\text{supp}(\theta)\subset B(0,5)$, $\|\dun \theta\|_2=1,\ \|\dun\theta\|_\infty\leq 
1/2$, set $A=\dun \theta$ and 
$\rho=1+p^2A(px_1,p^2x_2)$, $\phi=p\theta(px_1,p^2x_2)$. Since $\text{supp}(\rho)\cup \text{supp}
(\phi)\subset B(0,5/p^2)$, there is an obvious way to define them as functions $(\rho_n,\phi_n)\in 
C^\infty(\mathbb{T}_n^2)$ 
for any $n\geq 5/p^2$. Next using $p\leq 1$, $\chi(\rho)=\rho$, basic computations give
\begin{eqnarray*}
P_n(\theta,\varphi)&=&p,\\ 
\widetilde{E}_n^\varepsilon(\rho_n,\phi_n)&=&
\int_{\R^2} 
\frac{1+p^2A(z_1,z_2)}{2}(p(\dun \theta)^2+p^3(\partial_2\theta)^2)+\frac{K(\rho)}{2}\big(p^3
(\dun A)^2+p^5(\partial_2A)^2+\widetilde{G}\\
&&\hspace{1cm}+\frac{\varepsilon}{2}\bigg(p^3\big(\dun^2\theta+p^2\partial_2^2\theta\big)^2
+p^5\big(\dun^2A+\partial_2^2A\big)^2\bigg)\,dz\\
&&\leq p\bigg(\frac{3}{4}(1+\|\partial_2\theta\|_2^2)+\max_{[1/2,3/2]}K\frac{\|\dun A\|_2^2
+\|\partial_2A\|_2^2}{2}+\frac{\max \widetilde{G}''}{2}\\
&&\hspace{3cm}+\|\dun^2\theta\|_2^2+\|\partial_2^2\theta\|_2^2
+\|\partial_1^2A\|_2^2++\|\partial_2^2A\|_2^2\bigg)
\end{eqnarray*}
The constant in factor of $p$ is clearly independent of $p,n,\varepsilon$.
\end{proof}

\begin{lemma} \label{existsmooth}
For any $p>0$, $n\geq 1,\ \varepsilon>0$ the minimization problem 
\begin{eqnarray*}
\inf\bigg\{\widetilde{E}_n^\varepsilon(\rho,\phi)\text{ with } (\rho-1,\nabla\phi)\in H^2\times H^1,\
P_n(\rho,\phi)=\int_{\mathbb{T}_n^2}(\rho-1)\dun\phi\,dx=p\bigg\}.
\end{eqnarray*}
admits a minimizer $(\rho_n^\varepsilon(p)-1,\phi_n^\varepsilon(p))\in H^2\times H^2$, 
solution of 
\begin{equation}\label{ellsmooth}
\left\{
\begin{array}{ll}
\displaystyle 
-c_{n,\varepsilon}\partial_1\rho_n^\varepsilon +\text{div}(\chi(\rho_n^\varepsilon)
\nabla \phi_n^\varepsilon)-\varepsilon\Delta^2\phi_n^\varepsilon=0,\\
\displaystyle 
-c_{n,\varepsilon}\partial_1\phi_n^\varepsilon+\chi'(\rho_n^\varepsilon)\frac{|\nabla 
\phi_n^\varepsilon|^2}{2}-K(\chi(\rho_n^\varepsilon))\Delta\rho_n^\varepsilon
-\frac{(K\circ\chi)'|\nabla\rho_n^\varepsilon|^2}{2}+\widetilde{g}(\rho_n^\varepsilon)
+\varepsilon\Delta^2\rho_n^\varepsilon=0,
\end{array}
\right.
\end{equation}
Moreover, there exists $p_1,M$ such that for $p\leq p_1$, $n\geq 5/p^2$, $\varepsilon\leq 1$,
$$c_{n,\varepsilon}\leq M,\
\|\rho_n^\varepsilon(p)-1\|_{H^2}+\|\nabla\phi_n^\varepsilon(p)\|_{H^1}\leq Mp.$$
Furthermore for any 
$j\geq 2$, there exists $F_j(p)\rightarrow_0 0$ such that 
$$\|(\rho_n^\varepsilon-1,\nabla\phi_n^\varepsilon)\|_{H^j\times H^{j-1}}\leq F_j(p).$$

\end{lemma}
\begin{proof}
We follow the scheme of proof of proposition \ref{mainest} with a few technical additions.
For simplicity of notations, we drop the $\varepsilon,n$ indices. If $\rho_k,\varphi_k$ is a 
minimizing sequence, by weak compactness and proposition \ref{anfunc} 
we can assume $\rho_k-1\rightharpoonup\rho-1\ (H^2)$, $\nabla \varphi_k
\rightharpoonup \nabla \phi\ (H^1)$, and from Rellich's compact embedding we have 
\begin{eqnarray*}
\lim_k \int_{\mathbb{T}_n^2}\rho_k|\nabla \phi_k|^2+
K(\chi(\rho_k))|\nabla \rho_k|^2+\widetilde{G}(\rho_k)dx=\int_{\mathbb{T}_n^2}\rho|\nabla \phi|^2+
K(\chi(\rho))|\nabla \rho|^2+\widetilde{G}(\rho)dx,\\
p=\lim\int_{\mathbb{T}_n^2}(\rho_k-1)\dun\phi_kdx=\int_{\mathbb{T}_n^2}(\rho-1)\dun\phi\,dx.
\end{eqnarray*}
We combine it with lower semi-continuity to obtain $\underline{\lim}_k\widetilde{E}_n^\varepsilon
(\rho_k,\varphi_k)\geq \widetilde{E}_n^\varepsilon(\rho,\phi)$, so that $(\rho,\phi)$ is a 
minimiser and solves \eqref{ellsmooth} for some $c(n,p,\varepsilon)$.
By standard elliptic regularity, $(\rho,u)$ is smooth (with norms a priori depending on 
$\varepsilon$).
Multiplying the first equation by $\phi$ and integrating by parts, we find 
\begin{equation*}
c\int_{\mathbb{T}_n^2}(\rho-1)\dun \phi dx=\int \chi(\rho)|\nabla \phi|^2+\varepsilon 
|\Delta\phi|^2dx\leq 2(\widetilde{E}_{n}^\varepsilon)_{\min}(p)\Leftrightarrow 0<c\leq 
\frac{2(\widetilde{E}_n^\varepsilon)_{\min}(p)}{p}
\end{equation*}
We deduce $n\geq 5/p^2\Rightarrow c<2M$ with $M$ the constant of lemma \ref{estimeneps}. With this 
bound on $c(n,p,\varepsilon)$ we can now obtain uniform elliptic estimates. 
% We multiply the first 
% equation by $\Delta \phi$ and integrate, this gives 
% \begin{eqnarray*}
% \int_{\mathbb{T}_n^2}-c\dun \rho \Delta \phi+\n(\chi(\rho))\cdot \nabla \phi\,\Delta \phi+
% \chi(\rho)|\Delta\phi|^2+\varepsilon|\nabla \Delta \phi|^2dx=0\\
% \Rightarrow \frac{1}{4}\int_{\mathbb{T}_n^2} |\Delta\phi|^2dx\leq 2c\|\dun \rho\|_2^2+
% 2\|\nabla \chi(\rho)\|_4^2\|\nabla \phi\|_4^2.
% \end{eqnarray*}
% Next we multiply the second equation by $\Delta \rho$ and integrate
% \begin{eqnarray*}
% \int_{\mathbb{T}_n^2}K(\chi(\rho))(\Delta\rho)^2+\varepsilon|\n\Delta \rho|^2dx&=&
% \int_{\mathbb{T}_n^2}
% c \dun \phi \Delta\rho +\frac{1}{2}(K\circ\chi)'(\rho)|\nabla \rho|^2\Delta \rho
% +\widetilde{g}(\rho)\Delta \rho\,dx\\
% &\leq & \big(c\|\dun \phi\|_2+\|(K\circ\chi)'\|_\infty\|\n\rho\|_4^2+\|\rho-1\|_2\big)
% \|\Delta\rho\|_2
% \end{eqnarray*}
The same computations as for proposition \ref{mainest} lead to
\begin{eqnarray*}
\int_{\mathbb{T}_n^2}\chi |\Delta\phi|^2+\varepsilon |\nabla \Delta \phi|^2dx&\leq& 2c\|\dun \rho\|_2^2+
2\|\nabla \chi(\rho)\|_4^2\|\nabla \phi\|_4^2,\\
\int_{\mathbb{T}_n^2}K(\chi(\rho))(\Delta\rho)^2+\varepsilon|\n\Delta \rho|^2dx
&\leq&  \big(c\|\dun \phi\|_2+\|(K\circ\chi)'\|_\infty\|\n\rho\|_4^2+\|\rho-1\|_2\big)
\|\Delta\rho\|_2
\end{eqnarray*}
As in proposition \ref{mainest} we use $\|\nabla \psi\|_4^2\lesssim\|\nabla \psi\|_2\|\Delta\psi\|_2$
to get for some $C>0$ independent of $M,\varepsilon,n\geq 5/p^2$
\begin{equation*}
\|\Delta \rho\|_2^2+\|\Delta \phi\|_2^2\lesssim \|\dun \phi\|_2^2+\|\dun\rho\|_2^2
+\|\Delta \rho\|_2\|\rho-1\|_2+\|\nabla \phi\|_4^4+\|\nabla \rho\|_4^4
\end{equation*}
\begin{eqnarray*}
\Rightarrow 
\big(\|\Delta \rho\|_2^2+\|\Delta \phi\|_2^2\big)(1-C(\|\nabla\rho\|_2^2+\|\nabla \phi\|_2^2))
&\lesssim & \|\dun \phi\|_2^2+\|\dun\rho\|_2^2)+\|\rho-1\|_2^2\\
&\lesssim& \widetilde{E}_n^\varepsilon(\rho,\phi).
\end{eqnarray*}
Using $\widetilde{E}_n^\varepsilon\leq Mp$ from lemma \ref{estimeneps}, we obtain for $p<<1/C$
\begin{equation*}
\|\Delta \rho\|_2^2+\|\Delta \phi\|_2^2\leq M'p,\ M'\text{ indepent of }p,\varepsilon,n\geq 5/p^2.
\end{equation*}
The estimate for $j=2$ follows since the energy controls $\|\rho-1\|_{H^1}+\|\nabla \phi\|_{L^2}$, 
the case $j>2$ is a standard bootstrap argument.
\end{proof}

\begin{prop} \label{existtore}
Let $p_1$ as in lemma \ref{existsmooth}. For any $p\leq p_1$, $n\geq 5/p^2$, there exists 
$(\rho_n,\phi_n)\in C^\infty(\mathbb{T}_n^2)$ 
such that up to an extraction, for any $j\geq 1$, $\|\rho_n^\varepsilon-\rho_n\|_{H^j}
+\|\n\phi_n^\varepsilon-\n\phi_n\|_{H^{j-1}}\rightarrow_{\varepsilon\rightarrow 0} 0$, 
$(\rho_n,\phi_n)$ is a solution of the 
minimization problem 
\begin{eqnarray*}
\inf\bigg\{
\widetilde{E}_n(\rho,\phi)=\int_{\mathbb{T}_n^2}\frac{1}{2}(\chi(\rho)|\nabla \phi|^2+K(\chi(\rho))
|\nabla \rho|^2)+\widetilde{G}(\rho)dx,\\ 
P_n(\rho,\phi)=\int_{\mathbb{T}_n^2}(\rho-1)\dun\phi\,dx=p_0\bigg\}.
\end{eqnarray*}
Moreover, $(\rho_n,\phi_n)$ is a solution of 
\begin{equation}\label{elltore}
\forall\,x\in \mathbb{T}_n^2\left\{
\begin{array}{ll}
\displaystyle 
-c_n\partial_1\rho_n +\text{div}(\chi(\rho_n)\nabla \phi_n)=0,\\
\displaystyle 
-c_n\partial_1\phi_n+\frac{\chi'|\nabla \phi_n|^2}{2}-K\Delta\rho_n-\frac{K'
|\nabla\rho_n|^2}{2}+\widetilde{g}(\rho_n)=0.
\end{array}
\right.
\end{equation}
for some $0\leq c_n\leq M$, $M$ the constant from lemma \ref{existsmooth} independent of $p$.
\end{prop}
\begin{proof}
We fix $p\leq p_1$, 
$(\rho_n^\varepsilon(p),\varphi_n^\varepsilon(p))$ a minimizer. Using the a priori bounds,
Rellich's compactness theorem and diagonal extraction we can extract a sequence 
$\varepsilon_k\rightarrow 0$ with 
$$
c_n^\varepsilon\rightarrow_{\varepsilon} c_n\in [0,M],\ \forall\,j\geq 1,\ 
(\rho_n^{\varepsilon_k}(p)-1,\nabla\phi_n^{\varepsilon_k}(p))\rightarrow (\rho_n-1,
\nabla \phi_n)\ (H^j\times H^{j-1}).$$ 
Therefore we can pass to the limit in \eqref{ellsmooth}: since $(\rho_n^\varepsilon,\nabla\phi_n^\varepsilon)$
remains uniformly bounded in $H^4\times H^3$, the terms $\varepsilon\Delta^2\varphi_n^\varepsilon,
\varepsilon\Delta^2\rho_n^\varepsilon$ vanish, and $(\rho_n,\nabla \phi_n)$ is a solution of 
\eqref{TWmod}. Similarly, $\displaystyle\wE(\rho_n,\phi_n)=\lim_\varepsilon \wE_{\min}(p)$, $P(\rho_n,\phi_n)=p$.\\
% \begin{equation*}
% \lim_k (\widetilde{E}_n^{\varepsilon_k})_{\min}(p)=
% \lim_k\widetilde{E}_n^{\varepsilon_k}(\rho_n^{\varepsilon_k},\phi_n^{\varepsilon_k})
%  =\widetilde{E}_n^0(\rho_n,\phi_n),\ \lim P(\rho_n^{\varepsilon_k},\phi_n^{\varepsilon_k})=p.
% \end{equation*} 
% The $L^\infty$ bound follows from $\|\rho_n-1\|_\infty\lesssim \|\rho_n-1\|_{H^2}\leq Mp$.
% thus up to an other extraction we can assume that it converges to $c\leq 2M$, since 
% $\|\Delta^2\rho_n^\varepsilon\|_{H^2}+\|\Delta^2\phi_n\|_{H^2}$ is bounded uniformly in 
% $\varepsilon$, the terms $\varepsilon\Delta^2\rho_n^\varepsilon,\varepsilon 
% \Delta^2\phi_n^\varepsilon$ vanish and the convergence of the other terms follows from
% Sobolev's embedding. In particular, if $\|\rho_n-1\|_\infty$ is small enough so that
% $\chi(\rho_n)=\rho_n,\ \widetilde{g}(\rho_n)=
% g(\rho_n)$ and we recover equations \eqref{elltore}.\\
To check the minimization property, we prove now
$\displaystyle \lim_{\varepsilon\rightarrow 0}
(\widetilde{E}_n^{\varepsilon})_{\min}(p)=(\widetilde{E}_n)_{\min}(p)$. Clearly, it suffices to prove 
$\leq$.
Let $\delta>0$, $(\rho,\phi)\in \mathcal{H}$ 
such that $P_n(\rho_n,\phi_n)=p,\ \widetilde{E}_n(\rho,\phi)\leq (\widetilde{E}_n)_{\min}(p)+\delta$. 
By density of smooth functions, there exists $\rho_k,\phi_k\in C^\infty(\mathbb{T}_n^2)$ such 
that $\|\rho-\rho_k\|_{H^1}+\|\nabla \phi-\nabla \phi_k\|_2\rightarrow_k 0$. In particular 
$P_n(\rho_k,\phi_k)=p_k\rightarrow p$ and (for $k$ large enough so that $p_k\neq 0$) 
$\|\nabla \phi -\frac{p}{p_k}\nabla \phi_k \|_2\rightarrow_k0$. Using Lebesgue's dominated convergence 
theorem and up to an extraction 
\begin{equation*}
 \widetilde{E}_n(\rho_k,\frac{p}{p_k}\phi_k)\rightarrow \widetilde{E}_n(\rho,\phi),\ P_n(\rho_k,
 \frac{p}{p_k}\phi_k)=p.
\end{equation*}
Let us fix $k$ large enough for which 
$\widetilde{E}_n(\rho_k,\phi_k)\leq \widetilde{E}_n(\rho,\phi)+\delta$. Then for $\varepsilon(k,\delta)$ 
small enough 
$\widetilde{E}_n^{\varepsilon}(\rho_k,\frac{p}{p_k}\phi_k)\leq 
\widetilde{E}_n(\rho_k,\frac{p}{p_k}\phi_k)+\delta\leq \widetilde{E}_n(\rho,\phi)+2\delta\leq 
(\widetilde{E}_n)_{\min}(p)+3\delta$. Since $\delta$ is arbitrary it ends the argument.
\end{proof}
\begin{rmq}
Using the identity $c_np=\int \chi(\rho_n)|\nabla \phi_n|^2$, $c_n$ is actually positive rather than nonnegative, 
but this is not useful here.
\end{rmq}

\subsection{Convergence of minimizers as \texorpdfstring{$n\rightarrow \infty$}{ninf}}
We start with the following immediate consequence of lemma \ref{smoothmin}.
\begin{prop}\label{convemin}
For any $p\geq 0$, $\displaystyle \mathop{\overline{\lim}}_{n\rightarrow \infty}
(\widetilde{E}_{n})_{\min}(p)\leq \widetilde{E}_{\min}(p)$.
\end{prop}

This opens the path to the existence of minimizers on $\R^2$.
In this section, we consider a sequence of minimizers $(\rho_n,\phi_n)$ of momentum $p$ 
on $\mathbb{T}_n^2$.
We identify $\mathbb{T}_n^2$ as $\Omega_n=[-n\pi,n\pi]^2\subset \R^2$. For any function $\psi_n$ 
defined on $\mathbb{T}_n^2$, $K$ compact, by ``$\psi_n\rightarrow \psi$ on $K$'' we implicitly 
identify $\psi_n$ with the function defined on $\Omega_n$, $n$ large enough so that 
$K\subset \Omega_n$.
\begin{prop}\label{convprofile}
Let $p\leq p_2=\min(p_0,p_1)$, with $(p_0,p_1)$ from proposition \ref{estimenergie} and lemma 
\ref{existsmooth}, let $(\rho_n(p),\phi_n(p))$ be a minimizer of $\widetilde{E}_n$ of momentum $p$. Assume
\begin{equation}\label{assump}
\exists\,\delta>0\ :\ \forall\,n\geq 0,\ |\rho_n(0)-1|\geq \delta,
\end{equation}
then up to 
an extraction 
there exists $(\rho,\nabla \phi)\in  (\cap_j H^j)^2$  such that 
\begin{enumerate}
\item for any $j\geq 1$, any compact $K\subset \R^2$, 
$\|\rho_n-\rho\|_{H^j(K)}+\|\nabla\phi_n-\nabla \phi\|_{H^{j-1}}\rightarrow 0$, in particular 
$|\rho(0)-1|\geq \delta$. 
\item $(\rho,\phi)$ is a solution of \eqref{TWmod} for some $c=\lim c_n\in ]0,M]$, $M$ independent of $p$.
\item $P(\rho,\phi)>0$.
\end{enumerate}
\end{prop}
\begin{proof}
Items 1. and 2. follow from the same argument as for proposition $\ref{existtore}$. As for $n$ 
large enough $0\leq c_n\leq M$, we have 
$0\leq c\leq M$. However, because the convergence 
is only \emph{local} we can not pass to the limit in 
$\wE(\rho_n,\phi_n)$ and $P(\rho_n,\phi_n)$. \\
For item  3. we note that the assumption $|\rho_n(0)-1|\geq \delta$ implies $\rho(0)\neq 1$, and since $\rho$ 
is smooth $\int_{\R^2} \widetilde{G}(\rho)dx>0$. Since $(\rho,\nabla\phi)$
is a solution of \eqref{TWmod}, it satisfies the identity \eqref{poho3} which reads 
$$cP(\rho,\phi)=2\int \widetilde{G}(\rho)dx.$$
The right hand side is positive, and $c\geq 0$, therefore $c>0$ and $P(\rho,\phi)>0$. 
\end{proof}
\begin{prop}\label{mainprop}
In proposition \ref{convprofile}, up to a translation assumption \eqref{assump} is true and 
\begin{eqnarray}
P(\rho,\phi)&=&\lim_nP_n(\rho_n,\phi_n)=p,\\
\lim_n(\wE_n)_{\min}(p)&=&\lim_n \wE_n(\rho_n,\nabla \phi_n)=\wE(\rho,\nabla\phi).
\end{eqnarray}
\end{prop}
In view of proposition \ref{convemin}, this proposition implies the existence of a solution to 
\eqref{TWmod} which is a constrained minimizer to $\wE$.
The key is to forbid the following behaviours of the sequence $(\rho_n,\phi_n)$:
\begin{itemize}
 \item dichotomy: the minimizing sequence $(\rho_n,\phi_n)$ splits in (at least) two profiles 
 whose supports are more and more distant.
 \item spreading: the total energy ``far from the profiles'' does not converge to $0$, although 
 $\rho_n,\phi_n\rightarrow (1,0)$ uniformly.
\end{itemize}

\paragraph{Profile decomposition and proof of proposition \ref{mainprop}}
We note $d(\cdot,\cdot)$ the distance on the 
torus $\mathbb{T}_n^2$, the energy density 
$\widetilde{e}(\rho,\phi)=\frac{1}{2}\big(\chi(\rho)|\nabla\phi|^2+
K\circ\chi(\rho)|\nabla \rho|^2
+(\rho-1)^2\big)$ and the momentum density $p(\rho_n,\phi_n)=(\rho_n-1)\dun \phi_n$. 
For $x\in \mathbb{T}_n^2$, the set $B(x,r)$ is the ball in $\mathbb{T}_n^2$. \\
The key lemma is a modification of proposition $4.2$ and lemma $5.2$ in 
\cite{BGS}. For the convenience of the reader we include a proof in the appendix.
\begin{lemma}\label{profildec}
Let $(\rho_n,\phi_n)$ be a sequence of minimizers of $\wE_n$ of momentum $p$.
Up to an extraction, $c_n\longrightarrow c\in ]0,M]$, $M$ independent of $p$.\\ 
For $\delta \lesssim p^2$, we set $A_n^\delta=\{x:\ |\rho_n(x)-1|\geq \delta\}$. Up to an other extraction 
there exists a sequence of radiuses $(R_n^k)_{n\geq k\geq 1}$, $l\in \N^*$, $(x_n^i)_{1\leq i\leq l}\in 
(\mathbb{T}_n^ 2)^l$, $M_k\rightarrow \infty$
such that : 
\begin{itemize}
\item $\forall\,n\geq 1,\ 1\leq i\leq l,\ |\rho_n(x_n^i)-1|\geq \delta$.
\item For any $n\geq k$, $\inf_{i\neq j}d(x_n^i,x_n^j)\geq 10 R_n^k$ and
$A_n^\delta\subset \bigsqcup_{i=1}^lB(x_n^i,R_n^k)$.
\item There exists $C$ independent of $\delta,n,k$  such that for any $n\geq k$, $1\leq i\leq l$
\begin{equation}\label{controlspread}
\displaystyle \bigg|\int_{(\cup B(x_n^i,R_n^k))^c}\widetilde{e}(\rho_n,\phi_n)
-c_np(\rho_n,\phi_n)dx\bigg|\leq C\bigg(\delta  \int_{(\cup B(x_n^i,R_n^k))^c}
\widetilde{e}(\rho_n,\phi_n)dx+\frac{\widetilde{E_n}(\rho_n,\phi_n)}{M_k}\bigg).
\end{equation}
% \item $c_n\longrightarrow_n c\in ]0,M]$, $M$ independent of $\delta,p$.
\item $\forall\,k\geq 1,\ R_n^k\longrightarrow_n R^k<\infty$, $R^k\longrightarrow_k \infty$. 
\end{itemize}
% \item $\displaystyle \sum_{i=1}^l \int_{B(x_n^i,2R_k)\setminus B(x_n^i,R_k)}
% \widetilde{e}(\rho_n\phi_n)dx\leq 2E_n(\rho_n,\phi_n)/|\ln(\kappa_k)|$
Moreover, if $c<1$, for $\delta$ small enough we can replace  \eqref{controlspread} by 
\begin{equation}\label{vanishspread}
 \bigg|\int_{(\cup B(x_n^i,R_n^k))^c}\widetilde{e}(\rho_n,\phi_n)
 dx\bigg|\leq C\frac{\widetilde{E_n}(\rho_n,\phi_n)}{(1-c)M_k}\bigg).
\end{equation}

\end{lemma}
\begin{rmq}
Basically, the lemma states that there are two areas: several balls far from each other on which non 
trivial profiles persist as $n\rightarrow \infty$, and a rest where there may be 
some ``spreading'' contribution to the total energy, but which is alsmot equal to the spreading 
part of the momentum. If $c<1$ there is no spreading.\\
Note also that $l\in \N^*$ means that ``pure spreading'' does not occur.
\end{rmq}

The better estimate available if $c<1$ makes this case quite simpler. Actually a consequence of 
\eqref{csharp} is that $c\geq 1$ does not occur, unfortunately, the existence of minimizers is a prerequisite 
to this estimate.\\
An interesting alternative approach would have been to prove directly that there exists no solution to 
\eqref{TWmod} if $c\geq 1$, as was done in \cite{Gravejat} for the Gross-Pitaevskii case.
\paragraph{The case $c\geq 1$}
\subparagraph{The dichotomy scenario} In this paragraph, we show that the sequence of minimizers 
can not split in several profiles.
% Finally, we will use the subadditivity 
% \begin{equation}\label{subadd}
% \forall\,A\geq 0,\ \alpha(p_1,p_2)
% \leq \alpha(A+p_1,p_2):= \widetilde{E_{\min}}(A+p_1)
% +\widetilde{E_{\min}}(p_2)-\widetilde{E_{\min}}(A+p_1+p_2)
% \end{equation}

\begin{prop}
In lemma \ref{profildec}, we have for $\delta$ small enough $l=1$. 
\end{prop}
\begin{proof}
First we note that $l$ as a function of $\delta$ is nonincreasing, as the existence 
of $l$ points such that $d(x_n^i,x_n^j)\rightarrow \infty$ and $|\rho(x_n^i)-1|\geq \varepsilon$ 
prevents from covering $A_n^\delta$ for $\delta<\varepsilon$ by less than $l$ balls of radius bounded in $n$.
% if $l>1$ for some $\delta_0>0$, then $l>1$ for $0<\delta<\delta_0$: indeed 
% lemma \ref{profildec} gives
% $|\rho_n(x_n^1)-1|\geq \delta_0,\ |\rho_n(x_n^2)-1|\geq \delta_0$ and $d(x_n^1,x_n^2)
% \rightarrow \infty$
We assume by contradiction that there exists 
$\varepsilon>0$, such that $l(\varepsilon)\geq 2$. This implies the existence of 
$(y_n^1,y_n^2)\in \mathbb{T}_n^2$ such that $|\rho(y_n^i)-1|\geq \varepsilon,\ 
d(y_n^1,y_n^2)\longrightarrow_n \infty$. For $0<\delta\leq \varepsilon$, we can assume up to a reindexation
\begin{equation}\label{refineprofil}
\forall\,n\geq k,\ y_n^1\in B(x_n^1,R_n^k),\ y_n^2\in B(x_n^2,R_n^k).
\end{equation}
Since $\|\nabla \rho_n\|_\infty$ is bounded uniformly in $n$, we also remark
\begin{equation}\label{minorprofile}
\exists q>0\text{ independent of }\delta,\ 
\forall\,n\geq k\geq 1,\ \forall\,i=1,2,\ \int_{B(x_n^i,R_n^k)} \widetilde{G}(\rho_n)\geq q.
\end{equation}
We apply proposition \ref{convprofile} 
to $(\rho_n,\phi_n)(\cdot-x_n^i)$ : up to an extraction there exists $(\rho^i,\nabla \phi^i)\in 
\cap_{j\geq 0} H^j(\R^2)$, solutions of \eqref{TWmod} with speed $c$ and
\begin{equation*}
\forall\,K\text{ compact}, \forall\,j\geq 1,\ \|\rho_n(\cdot-x_n^i)-\rho^i\|_{H^j(K)}+\|\nabla 
\phi_n(\cdot-x_n^i)-\nabla\phi^i\|_{H^{j-1}(K)}\rightarrow 0.
\end{equation*}
% \begin{equation*}
% \Rightarrow \left\{\begin{array}{ll}
% \lim_{k\rightarrow \infty}\lim_{n\rightarrow \infty}\int_{B(x_n^i,R_n^k)}
% \widetilde{e}(\rho_n,\phi_n)dx=\widetilde{E}(\rho_i,\phi_i),\\ 
% \lim_{k\rightarrow \infty}\lim_{n\rightarrow \infty}\int_{B(x_n^i,R_k)}p(\rho^n,\phi^n)dx=
% P(\rho_i,\phi_i)
% \end{array}\right.
% \end{equation*}
We can assume that $\widetilde{E_n}(\rho^n,\phi^n)$ converges and thus
\begin{equation}\label{spreadloc}
\forall\,k\geq 1,\ \exists\,(\mu_k,\nu_k)\in \R^+\times \R:
\left\{
\begin{array}{ll}
\displaystyle \lim_n \int_{\mathbb{T}_n^2}\widetilde{e}(\rho_n,\phi_n)dx=\sum_{i=1}^l\int_{B(0,R^k)}
\widetilde{e}(\rho^i,\phi^i)dx+\mu_k,\\
\displaystyle \lim_n \int_{\mathbb{T}_n^2}p(\rho_n,\phi_n)dx=\sum_{i=1}^l\int_{B(0,R^k)}
p(\rho^i,\phi^i)dx+\nu_k,\\
\displaystyle |\mu_k-c\nu_k|\leq C\bigg(\delta \mu_k+\frac{\lim_n\widetilde{E_n}(\rho_n,\phi_n)}{M_k}\bigg),
\end{array}
\right.
\end{equation}
where $B(0,R^k)$ is now the usual ball of $\R^2$.
Letting $k\longrightarrow \infty$ implies
\begin{equation*}
\lim_n \widetilde{E_n}(\rho^n,\phi^n)=\sum_{i=1}^l\widetilde{E}(\rho^i,\phi^i)+\mu,\ \lim_n P_n(\rho_n,\phi_n)=
\sum_{i=1}^l P(\rho^i,\phi^i)+\nu,\ |\mu-c\nu|\leq C\delta \mu.
\end{equation*}
Since $C$ is an absolute constant, we can assume $C\delta<1$, thus $\nu\geq 0$.
Let us set $p_i:=P(\rho^i,\phi^i)$. Since $(\rho_i,\nabla \phi_i)$ is a solution of \eqref{TWmod}, 
identity \eqref{poho3} is true, namely 
\begin{equation*}
cp_i=2\int_{\R^2}\widetilde{G}(\rho_i)dx>0,
\end{equation*}
thus $p_i>0$ and from \eqref{minorprofile} $p_1\geq 2q/M,\ p_2\geq 2q/M$. On the other hand, we know 
that $\forall\,n,\ P_n(\rho_n,\phi_n)=p$, so that by subadditivity and proposition \ref{convemin}
\begin{eqnarray*}
\sum_{i=1}^l\widetilde{E}(\rho^i,\phi^i)+\mu=
\lim_n \widetilde{E_n}(\rho_n,\phi_n)\leq \widetilde{E}_{\min}(p)
=\widetilde{E}_{\min}\bigg(\sum_{i=1}^lp_i+\nu\bigg)\\
\Rightarrow 
\widetilde{E}(\rho^1,\phi^1)+\widetilde{E}(\rho^2,\phi^2)+\sum_{i=3}^l\widetilde{E}_{\min}(p_i)\leq 
\widetilde{E}_{\min}\bigg(\sum_{i=1}^lp_i\bigg)+\widetilde{E}_{\min}(\nu)-\mu
\end{eqnarray*}
Next we use proposition \ref{estimenergie}: $\widetilde{E}_{\min}(\nu)\leq \nu\leq c\nu$ , 
subadditivity and $|\mu-c\nu|\leq C\delta \mu$
\begin{eqnarray*}
\wE_{min}(p_1)+\wE_{min}(p_2)\leq
\widetilde{E}(\rho^1,\phi^1)+\widetilde{E}(\rho^2,\phi^2)\leq 
\widetilde{E}_{\min}\bigg(p_1+p_2\bigg)+C\delta \mu
% \leq \sum_{i=1}^l\widetilde{E}_{\min}(p_i)-\alpha(q,q)+C\delta\mu\\
% \Rightarrow 0\leq C\delta\mu-\alpha(q,q).
\end{eqnarray*}
% where $\mu\leq \widetilde{E}_{\min}(p)$ and $C$ is an absolute constant. From \eqref{minorprofile},
% we have for $i=1,2$, $\widetilde{E}(\rho^i,\phi^i)\geq q$ and $p_i\geq q/M$. Now let us recall that $p_1,p_2$ depend 
% on $\delta$, but not $q$. 
But from proposition \ref{posD}, 
$\wE_{min}(p_1)+\wE_{min}(p_2)-\widetilde{E}_{\min}\bigg(p_1+p_2\bigg)\geq D(2q/M,2q/M)>0$, while letting 
$\delta\rightarrow 0$ we find
\begin{equation}
D(2q/M,2q/M)\leq  0,
\end{equation}
which is a contradiction.

\end{proof}
\subparagraph{The spreading scenario} Ruling out this scenario follows the same scheme, but simpler.
Since $l=1$, from the same computations as in the previous paragraph for any $\delta>0$ there exists 
$(\rho,\phi),\ \mu\geq 0,\ p_1>0$ such that $P(\rho,\phi)=p_1$ and 
\begin{equation*}
\lim_n (\wE_n)_{\min}(\rho_n,\phi_n)=\wE(\rho,\phi)+\mu,\ \lim_nP_n(\rho_n,\phi_n)=p_1+\nu,\ |c\nu-\mu|\leq C\delta \mu.
\end{equation*}
We use $\wE_{\min}(\nu)\leq \nu\leq c\nu\leq \mu+C\delta \mu$ so that
$$
E_{\min}(p_1)+E_{\min}(\nu)\leq E_{\min}(p_1)+\mu+C\delta \mu= 
\lim_nE_n(\rho_n,\phi_n)+C\delta \mu\leq E_{\min}(p_1+\nu)+C\delta \mu.
$$
proposition \ref{posD} with $q=\min(\nu,p_1)$ implies
$0\leq -D(q,q)+C\delta\mu$, letting $\delta\rightarrow 0$ we get $q=0$, thus $\mu=\nu=0$.
\subparagraph{Conclusion}
We have obtained that there exists $(\rho,\nabla \phi)\in (\cap_{j\geq 0}H^j)^2$ such that 
\begin{eqnarray*}
\forall\,K\text{ compact}, \|\rho_n-\rho\|_{H^j(K)}+\|\nabla\phi_n-\nabla \phi\|_{H^j(K)}
\longrightarrow _n 0,\\
\lim_n \widetilde{E}_n(\rho_n,\phi_n)=E(\rho,\phi),\ 
p=\lim_n P_n(\rho_n,\phi_n)=P(\rho,\phi).
\end{eqnarray*}
this ends the proof of \ref{mainprop} in the case $c\geq 1$.

\paragraph{The case $c<1$} With the same notations as in the case $c\geq 1$ we have the existence 
of $(\rho^i,\nabla\phi^i)_{1\leq i\leq l}$ such that 
$\|\rho_n(\cdot-x_n^i)-\rho^i\|_{H^j(K)}\longrightarrow 0,\ \|\nabla\phi_n^i(\cdot-x_n^i)
-\nabla\phi^i\|_{H^j(K)}\longrightarrow 0$. Let us fix $\delta$ small enough so that 
inequality \eqref{vanishspread} is true. Thanks to the pointwise inequality 
$|(\rho-1)\partial_1\phi|\lesssim C\widetilde{e}(\rho,\phi)$ we get the following identities 
\begin{eqnarray*}
\lim_n\int \widetilde{E_n}(\rho_n,\phi_n)&=& \lim_k\sum_{i=1}^l\int_{B(0,R^k)}\widetilde{e}(\rho^i,
\phi^i)dx+O(1/M_k)=\sum_{i=1}^l\widetilde{E}(\rho^i,\phi^i),\\
p&=&\lim_n\int P_n(\rho_n,\phi_n)dx= \sum_{i=1}^lP(\rho^i,\phi^i).
\end{eqnarray*}
For $1\leq i\leq l$, set $p_i=P(\rho_i,\phi_i)$.  If for some $\delta>0$, $l\geq 2$, then 
we have as for the case $c\geq 1$ 
\begin{eqnarray*}
\widetilde{E}_{\min}(p_1)+\widetilde{E}_{\min}(p_2)+\sum_{i=3}^l\widetilde{E}_{\min}(p_i)\leq
\widetilde{E}_{\min}(p_1+p_2)+\sum_{i=1}^3\widetilde{E_{\min}}(p_i),
\end{eqnarray*}
which leads to the absurd inequality $0\leq -D(p_1,p_2)$. Thus $l=1$, the conclusion is the same as for $c\geq 1$.

\section{Pohozaev type identities and applications}\label{poho}
In this section we complete the proof of theorem \ref{maintheo} with the sharp estimates on 
the energy near $p=0$.\\
The first proposition does not rely on the fact 
that the dimension $d$ is $2$, therefore we state it in general settings. Since 
the solutions to \eqref{TWmod} that we constructed in the previous section are smooth
we state our identities for smooth functions, but they 
are true under much weaker assumptions.\\
For conciseness we write $K$ for $K(\chi(\rho))$, $K'=d(K\circ \chi)/d\rho$.
\begin{prop}\label{proppoho}
Let $(\rho,\phi)$ be a smooth finite energy solution of \eqref{TWmod}. If $(\rho-1,\phi)\in (H^2)^2$, 
then it satisfies the Pohozaev identities
\begin{eqnarray}\label{poho1}
\widetilde{E}(\rho,\phi)&=&\int_{\R^d} \chi|\dun \phi|^2+K|\dun \rho|^2dx,\\
\label{poho2}
\forall\,2\leq j\leq d,\ \widetilde{E}(\rho,\phi)&=&\int_{\R^d}
\chi|\partial_j \phi|^2+K|\partial_j \rho|^2dx+cP(\rho,\phi),\\
\label{poho3}
\frac{d-2}{2}\int_{\R^d}\chi |\nabla \phi|^2+K|\nabla \rho|^2dx&=&-d\int_{\R^d}\widetilde{G}(\rho)
dx+(d-1)cP(\rho,\phi).
\end{eqnarray}
Moreover we have 
\begin{equation}\label{energie1}
cP(\rho,\phi)=\int_{\R^d} \chi|\nabla \phi|^2dx
\end{equation}
\end{prop}
\begin{proof}
Multiply the first equation of \eqref{TWmod} by $x_1\partial_1\phi$ and integrate (note that 
the integrals are not clearly convergent, for a rigorous argument see e.g. proposition 5 in 
\cite{Gravejat}):
\begin{eqnarray}\nonumber
\int_{\R^d} -cx_1\dun \rho\dun \phi -\chi\nabla \phi \cdot \nabla(x_1\dun \phi)dx
&=&\int_{\R^d} -cx_1\dun \rho\dun \phi -\chi |\dun\phi|^2 -\chi x_1\dun \frac{|\nabla \phi|^2}
{2}dx\\
\label{pohoz1}
\nonumber 
&=&\int_{\R^d} -cx_1\dun \rho \dun \phi -\chi |\dun\phi|^2 +\frac{\chi |\nabla \phi|^2}{2}
\\
&&\hspace{4cm}+\frac{x_1\dun \chi |\nabla \phi|^2}{2}dx\\
\nonumber&=&0
\end{eqnarray}
Now the multiplication of the second equation of \eqref{TWmod} by $x_1\dun \rho$ and integration 
gives
\begin{eqnarray*}
0=\int_{\R^d} -cx_1\dun \phi\dun \rho +\frac{x_1\chi'\dun \rho|\nabla \phi|^2}{2}
&-&\bigg(K\Delta \rho+\frac{1}{2} K'|\nabla \rho|^2\bigg)x_1\dun\rho+\widetilde{g}(\rho)x_1
\dun \rho\,dx,
\\
\text{with} \int_{\R^d}-K\Delta \rho x_1\dun \rho\, dx&=& \int_{\R^d}K|\dun \rho|^2
+x_1K'|\nabla\rho|^2\dun \rho +x_1K\dun\bigg(\frac{|\nabla \rho|^2}{2}\bigg)\,dx\\
&=&\int_{\R^d}K|\dun \rho|^2+\frac{x_1K'|\nabla\rho|^2\dun \rho}{2} 
-\frac{K|\nabla \rho|^2}{2}\,dx,\\
\text{and }\int_{\R^d}x_1\widetilde{g}(\rho)\dun \rho&=&\int_{\R^d}-\widetilde{G}(\rho)dx,
\end{eqnarray*}
so that
\begin{equation}\label{pohoz2}
0=\int_{\R^d}-cx_1\dun \phi\dun \rho +\frac{x_1\chi'\dun \rho |\nabla \phi|^2}{2}+K|\dun \rho|^2-
\frac{K|\nabla \rho|^2}{2}-\widetilde{G}(\rho)dx.
\end{equation}
Finally, if we add \eqref{pohoz2} to \eqref{pohoz1} we obtain \eqref{poho1}
\begin{equation*}
0=\int_{\R^d}\chi\frac{|\nabla \phi|^2}{2}+K\frac{|\nabla \rho|^2}{2}+\widetilde{G}(\rho)
-\chi|\dun\phi|^2
-K|\dun\rho|^2.
\end{equation*}
The same computations with multipliers $x_j\partial_j\rho$ and $x_j\partial_j\phi$, $j\geq 2$ 
lead to 
\begin{equation*}
0=\int \frac{\chi|\nabla \phi|^2+K|\nabla \rho|^2}{2}+\widetilde{G}(\rho)
-\chi|\partial_j \phi|^2-K|\partial_j
\rho|^2-c\big(\dun \rho x_j\partial_j\phi-c\partial_j\rho x_j\dun \phi\big)dx.
\end{equation*}
This gives \eqref{poho2}, indeed an integration by part shows
\begin{eqnarray*}
\int \dun \rho x_j\partial_j\phi-c\partial_j\rho x_j\dun \phi\big)dx&=&
\int -x_j(\rho-1)\partial_1\partial_j\phi+x_j(\rho-1)\partial_1\partial_j\phi
+(\rho-1)\partial_1\phi dx\\
&=&P(\rho,\phi).
\end{eqnarray*}
The third identity is obtained by summing the previous ones. The last identity is obtained by 
multiplying the first equation in \eqref{TWmod} by $\phi$ and integration.
\end{proof}

\begin{prop}\label{estimvitesse}
Let $p_0>0$ given by prop \ref{estimenergie} $p_0,$ $\alpha,\beta$ positive
such that for any (smooth) minimizer of speed $c$ and momentum $p\leq p_0$, 
$$\alpha p^2\leq 1-c\leq \beta p^2.$$
\end{prop}
\begin{proof}
For $p\leq p_0$, let $(\rho,\phi)$ be such a minimiser. From \eqref{poho3}, 
\eqref{energie1} and proposition 
\ref{estimenergie}
$$cp=\frac{1}{2}\int 2\widetilde{G}+\chi|\nabla \phi|^2dx
\leq \widetilde{E}_{\min}(p)\leq p-\alpha p^3$$
which gives $1-c\geq \alpha p^2$. The other inequality follows an idea
from \cite{BGS}: applying $\dun$ to the first equation in \eqref{TWmod} gives 
\begin{equation*}
-c\dun^2\rho +\dun \Delta \phi +\dun \text{div}((\chi-1)\nabla\phi)=0.
\end{equation*}
Next, multiply the momentum equation by $K(1)/K$ and apply $\Delta$ :
\begin{equation*}
-c\Delta \dun \phi -K(1)\Delta^2 \rho+\Delta \rho
+\Delta\bigg(c\bigg(1-\frac{K(1)}{K}\bigg)\dun \phi+\frac{K(1) g}{K}-\rho 
+\frac{K(1)\chi'}{2K}|\nabla \phi|^2-\frac{K(1) K'}{2K}|\nabla \rho|^2\bigg)=0.
\end{equation*}
if we add these equalities we obtain
\begin{eqnarray*}
(K(1)\Delta^2-\Delta+c\dun^2)\rho&=&\Delta\bigg(c\bigg(\frac{K(1)}{K}-1\bigg)\dun \phi
+\frac{K(1) g}{K}-\rho 
+\frac{K(1)\chi'}{2K}|\nabla \phi|^2\\
&&\hspace{3.5cm}-\frac{K(1) K'}{2K}|\nabla \rho|^2\bigg)+c\dun\text{div}((\chi-1)
\nabla\phi)\\
&:=& \Delta A+c\partial_1\text{div}B.
\end{eqnarray*}
As $\chi(\rho)$ is bounded, $\widetilde{g}'(1)=1$ and 
$K(\chi(1))=K(1)$, it is easy to see
$$
\|A\|_{L^1}+\|B\|_{L^1}\lesssim \widetilde{E}(\rho,\phi)=\widetilde{E}_{\min}(p).
$$
Since the Fourier transform maps continuously $L^1$ to $L^\infty$, we deduce 
\begin{eqnarray*}
\|\rho-1\|_2=2\pi\|\widehat{\rho-1}\|_2 &\leq& 
C(\|A\|_{L^1}+\|B\|_{L^1})
\bigg\|\frac{|\xi|^2+|\xi_1||\xi|}
{K(1)|\xi|^4+|\xi|^2-c|\xi_1|^2}\bigg\|_2
\\
&\leq & C\widetilde{E}_{\min}(p)\bigg\|\frac{|\xi|^2+|\xi_1||\xi|}
{K(1)|\xi|^4+|\xi|^2-c|\xi_1|^2}\bigg\|_2.
\end{eqnarray*}
As $c\leq 1-\alpha p^2<1$ the $L^2$ norm on the right hand side is finite, and an elementary 
explicit computation (see \cite{BGS} claim 2.59) gives 
\begin{equation*}
\|\rho-1\|_2^2\lesssim \frac{\widetilde{E}_{\min}^2}{\sqrt{1-c}}.
\end{equation*}
On the other hand we have from proposition \ref{proppoho} the \eqref{propgtilde}
$\int (\rho-1)^2dx\sim 2\int \widetilde{G}(\rho)dx=cp=\int \chi(\dun\phi)^2dx$, we deduce
\begin{equation*}
p=\int_{\R^2}(\rho-1)\dun \phi\lesssim \int_{\R^2}\widetilde{G}(\rho)+\frac{\chi\dun \phi^2}{2}
dx=cp,
\end{equation*}
so that $c\gtrsim 1$. Next $\int (\rho-1)^2dx\sim cp\geq c\widetilde{E}_{\min}(p)
\gtrsim \widetilde{E}_{\min}(p)$ 
and we can conclude
\begin{equation*}
\sqrt{1-c}\leq C\sqrt{2}\widetilde{E}_{\min}(p)\leq C\sqrt{2}p\Rightarrow c-1\gtrsim -p^2.
\end{equation*}

\end{proof}

\begin{coro}\label{Esharp}
There exists $p_0>0$, such that for  $p\leq p_0$, if there exists a minimiser of momentum $p$,
\begin{equation}\label{eqEsharp}
p-\beta p^3\leq \widetilde{E}_{\min}(p)\leq p-\alpha p^3.
\end{equation}
with the same $\alpha,\beta$ as in proposition \ref{estimvitesse}.
\end{coro}
\begin{proof}
The inequality $\widetilde{E}_{\min(p)}\leq p-\alpha p^3$ is proposition \ref{estimenergie}. 
Conversely thanks to propositions \ref{proppoho} and \ref{estimvitesse}
\begin{equation*}
\widetilde{E}(\rho,\phi)\geq \int \frac{\chi|\nabla \phi|^2}{2}+\widetilde{G}(\rho)dx=cp\geq p-\beta p^3.
\end{equation*}
\end{proof}
\begin{rmq}
Corollary \ref{Esharp} is rather natural with the following heuristic : 
consider the formal relation $\displaystyle \delta \widetilde{E}=c\delta P\Rightarrow 
\frac{dE_{\min}}{dp}=c$. If this was true corollary \ref{Esharp} would merely be 
a consequence of the integration in $p$ of the estimates on $c$. 
\end{rmq}

\appendix
\section{Proof of the existence of the profile decomposition}\label{appendiceprof}
This section is devoted to the proof of lemma \ref{profildec}. 
First, we recall that $c_n$ is bounded, up to an extraction we assume $c_n\rightarrow c>0$ (for the 
sign of $c$, see prop \ref{convprofile}). 
\\
According to proposition \ref{convemin} and proposition \ref{estimenergie}, 
$\underline{\lim}(\widetilde{E}_n)_{\min}(p)\leq \widetilde{E}_{\min}(p)\leq 1-\alpha p^2$. 
Therefore for $n$ large enough, $\widetilde{E}_n)_{\min}(p)\leq 1-\alpha p^2/2$, and a
straightforward modification of proposition \ref{estimenergie} implies
$\|\rho_n-1\|_{L^\infty(\mathbb{T}_n^2)}\gtrsim p^2$ This ensures that 
$A_n^\delta=\{|\rho_n-1|\geq \delta\}$ is not  empty at least for $\delta\lesssim p^2$ and $n$ 
large enough. Next for any $n\geq 0$, the set $A_n^\delta$ is compact, thus there exists a finite 
covering $\cup_{i=1}^{l(n)} B(x_n^i,1/3)\supset A_n^\delta$ such that $|\rho(x_n^i)-1|\geq \delta$.
Using Vitali's lemma, there is a
subset $J_n\subset \{1,\cdots,l\}$  such that for $i,j\in J_n$, 
$B(x_n^i,1/3)\cap B(x_n^j,1/3)=\emptyset$ and 
$\displaystyle \bigcup_{i\in J_n} B((x_n^i,1)\supset A_n^\delta$. From lemma \ref{existsmooth}
$\|\rho_n-1\|_{W^{1,\infty}}$ is bounded uniformly in $n$, then 
$$\frac{|J_n|\delta^2}{\|\rho_n-1\|_{W^{1,\infty}}}\lesssim \sum\int_{B(x_n^i,1/3)}(\rho_n-1)^2dx\lesssim 
\widetilde{E_n}(\rho_n,\phi_n),$$
so $|J_n|$ must be bounded uniformly too. Up to an extraction, we can assume that $|J_n|$ 
is a constant $l$. 
\\
There are two key lemmas. The first one is a kind of improved Vitali's lemma, stating that the ball 
can be chosen very far away from each other.
\begin{lemma}\label{coversepar}
Given a collection $\displaystyle\bigcup_{i=1}^l B(x_i,R)\subset \mathbb{T}_n^2$, for any 
$M\geq 2$, there exists a subset 
$J\subset \{1,\cdots,l\}$ and $R\leq R'\leq (2M)^lR$ such that $\sqcup_{j\in J}B(x_j,R')\supset 
\cup_{i=1}^lB(x_i,R)$ and for any $(j,k)\in J^2,\ d(x_j,x_k)\geq MR'$.
\end{lemma}
\noindent For the proof, we refer to \cite{BGS} lemma $4.12$.\\
The second lemma looks a lot like proposition $4.2$ from \cite{BGS}. We include a proof since 
there  is a few non trivial differences. Let us first fix some notations :
for fixed $n$, $R\geq 1$, $M>>1$, we apply lemma \ref{coversepar} to $\cup B(x_n^i,R)$. Up to 
reindexing, there exists $l_n\leq l$, $(x_n^i)_{1\leq i\leq l_n}$
$\sqcup_{i=1}^{l_n} B(x_n^i,R'_n)\supset \cup_{i=1}^l B(x_n^i,R)\supset A_n^\delta$, 
$d(x_n^i,x_n^j)\geq R'M$. 
\begin{lemma}\label{lemmespread}
If $\cup_{i=1}^{l_n} B(x_n^i,R')\supset A_n^\delta$ is as in lemma \ref{coversepar}, 
and $k\in \N^*$ such that $2^{k}<M/2$, there exists $1\leq m\leq k$, $C$ 
an absolute constant such that setting
$$
\mathcal{S}_n^k:=(\cup_{i=1}^{l_n} B(x_n^i,2^{m}R'))^c,
$$
then 
\begin{equation}\label{estimspread}
\bigg|\int_{\mathcal{S}_n^k} c_np(\rho_n,\phi_n)-
\widetilde{e}(\rho_n,\phi_n)dx\bigg|\leq C\bigg(\delta \int_{\mathcal{S}_n^k} 
\widetilde{e}(\rho_n,\phi_n)dx+\frac{\widetilde{E}_n(\rho_n,\phi_n)}{k}\bigg).
\end{equation}
\end{lemma}

\begin{proof}
We first remark that since the balls $B(x_n^i,2^kR')$ are disjoint, 
\begin{equation*}
\int_{\cup_{i=1}^{l_n} B(x_n^i,2^kR')\setminus B(x_n^i,R')}\widetilde{e}\,dx=\sum_{p=1}^{k}
\int_{\cup_{i=1}^{l_n} B(x_n^i,2^{p}R')\setminus B(x_n^i,2^{p-1}R')}\widetilde{e}\,dx
\end{equation*}
In particular, there exists $1\leq m\leq k$ such that 
\begin{equation}\label{estimcouronne}
\int_{\cup_{i=1}^{l_n} B(x_n^i,2^{m}R')\setminus B(x_n^i,2^{m-1}R')}\widetilde{e}\,dx
\leq \frac{\int_{\cup_{i=1}^{l_n} B(x_n^i,2^{k}R')\setminus B(x_n^i,R')}\widetilde{e}\,dx}{k}
\leq \frac{\widetilde{E_n}(\rho_n,\phi_n)}{k}.
\end{equation}
Now let $\psi\in C_c^\infty(\R^+)$ such that $\psi|_{[0,1]}=1$, $\text{supp}(\psi)\subset [0,2]$, 
we define $\check{\phi}_n$ by 
\begin{eqnarray*}
\check{\phi}_n&=&\phi_n,\ x\in (\cup B(x_n^i,2^mR'))^c,\\
\check{\phi}_n&=&\psi\bigg(\frac{|x-x_n^i|}{2^{m-1}R'}\bigg)
\fint_{B(x_n^i,2^mR') \setminus B(x_n^i,2^{m-1}R')} \phi_ndx+(1-\psi) \phi_n,\ x\in B(x_n^i,2^mR').
\end{eqnarray*}
We recall the Poincar\'e-Wirtinger inequality
\begin{equation*}
\int_{B(0,2)\setminus B(0,1)}|f-\fint f|^2dx\leq C \|\nabla f\|_{L^2(B(0,2)\setminus B(0,1))}^2,
\end{equation*}
so that from a scaling argument 
\begin{equation}\label{PW}
\int_{B(x_n^i,2^mR') \setminus B(x_n^i,2^{m-1}R')}|\nabla \check{\phi_n}|^2dx\lesssim  \|\nabla \phi_n\|_
{L^2(B(x_n^i,2^mR') \setminus B(x_n^i,2^{m-1}R')}^2.
\end{equation}
If we multiply the first equation of \eqref{elltore} by $\check{\phi}_n$ and integrate over 
$\mathbb{T}_n^2$ we obtain 
\begin{eqnarray*}
\int_{\mathcal{S}_n^k}c_np(\rho_n,\phi_n)dx-\chi(\rho_n)|\nabla \phi_n|^2dx+
\sum_{i=1}^{l(n)}\int_{B(x_n^i,2^mR')\setminus B(x_n^i,2^{m-1}R')}c_n(\rho_n-1)\dun \check{\phi}_n\\
\hspace{1cm}-\chi(\rho_n)\nabla \phi_n\nabla\check{\phi}_ndx=0.
\end{eqnarray*}
Using Cauchy-Schwarz's inequality, \eqref{estimcouronne} and \eqref{PW} we can bound the second 
term
\begin{eqnarray*}
\big|\sum_{i=1}^{l(n)}\int_{B(x_n^i,2^mR')\setminus B(x_n^i,2^{m-1}R')}c_n(\rho_n-1)\dun 
\check{\phi}_n-\chi \nabla \phi_n\nabla\check{\phi}_ndx\big|
\lesssim \int_{\cup B(x_n^i,2^mR')\setminus B(x_n^i,2^{m-1}R')}\widetilde{e}(\rho_n,\phi_n)dx
\end{eqnarray*}
\begin{equation*}\hspace{6cm}
\leq \frac{\widetilde{E}_n(\rho_n,\phi_n)}{k}.
\end{equation*}
We have obtained 
\begin{equation}\label{idmasse}
\bigg|\int_{\mathcal{S}_n^k}c_np(\rho_n,\phi_n)-\rho_n|\nabla\phi_n|^2dx\bigg|
\leq C\frac{\widetilde{E}_n(\rho_n,\phi_n)}{k}.
\end{equation}
We turn to symmetric computations on the second equation of \eqref{elltore}. We set 
\begin{equation*}
\check{\rho}_n=
\left\{
\begin{array}{ll}
\displaystyle\rho_n,\ x\in (\cup B(x_n^i,2^mR'))^c,\\
\displaystyle\check{\rho}_n=\psi\bigg(\frac{|x-x_n^i|}{2^{m-1}R'}\bigg)
+(1-\psi) \rho_n,\ x\in B(x_n^i,2^{m}R').
\end{array}
\right.
\end{equation*}
In this case, since $(\rho_n-1)^2\lesssim \widetilde{e}_n(\rho_n,\phi_n)$ and 
$|\check{\rho}_n-1|\leq |\rho_n-1|$ we will not need the 
Poincar\'e-Wirtinger inequality. 
% \begin{equation*}
% \check{\rho}_n=
% \left\{
% \begin{array}{ll}
% \displaystyle\rho_n,\ x\in (\cup B(x_n^i,2^pR'))^c,\\
% \displaystyle\check{\rho}_n=\chi\bigg(\frac{|x-x_n^i|}{2^{p-1}R'+1}\bigg)
% \fint_{B(x_n^i,2^pR') \setminus B(x_n^i,2^{p-1}R')} \rho_ndx+(1-\chi) \rho_n,\ \\
% \hspace{4cm}x\in B(x_n^i,2^pR')\setminus B(x_n^i,2^{p-1}R'+\varepsilon),\\
% \displaystyle \chi\big((|x-x_n^i|-2^{p-1}R')/\varepsilon+1\big)+(1-\chi)\fint \rho_n\,dx,\ \\
% \hspace{4cm}x\in B(x_n^i,2^{p-1}R'+\varepsilon)\setminus B(x_n^i,2^{p-1}R'),\\
% 1,\ x\in B(x_n^i,2^{p-1}R').
% \end{array}
% \right.
% \end{equation*}
As in the previous section, we denote $K$ for $(K\circ \chi) (\rho_n)$, $K'=dK\circ \chi/d\rho$.
Multiplying the second 
equation of \eqref{elltore} by $\check{\rho}_n-1$ and integrating on $\mathbb{T}_n^2$ gives 
\begin{eqnarray*}
\int_{\mathbb{T}_n^2}-c_n(\check{\rho}_n-1)\dun \phi_n+\frac{(\check{\rho}_n-1)
|\nabla \phi_n|^2}{2} +K\nabla\rho_n \nabla \check{\rho}_n
+\frac{1}{2}K'|\nabla\rho_n|^2(\check{\rho}_n-1)\\
+\widetilde{g}(\rho_n)(\check{\rho}_n-1)dx=0
\end{eqnarray*}
We point out that $\check{\rho}_n-1=0$ on $\cup B(x_n^i,2^{m-1}R')$, thus 
\begin{eqnarray}
\nonumber
\int_{\cup B(x_n^i,2^mR')\setminus B(x_n^i,2^{m-1}R')}-c_n(\check{\rho}_n-1)\dun \phi_n
+\frac{(\check{\rho}_n-1)|\nabla \phi_n|^2}{2} +K\nabla\rho_n \nabla \check{\rho}_n
\\
\nonumber +\frac{1}{2}K'|\nabla\rho_n|^2(\check{\rho}_n-1)
+\widetilde{g}(\rho_n)(\check{\rho}_n-1)dx\\
\nonumber= -\int_{\mathcal{S}_n^k}-c_np(\rho_n,\phi_n)
 +K|\nabla\rho_n|^2+\widetilde{g}(\rho_n)(\rho_n-1)
\\
\label{step1}
\hspace{3mm}+\frac{(\rho_n-1)|\nabla \phi_n|^2}{2}+\frac{1}{2}K'|\nabla\rho_n|^2(\rho_n-1)dx
\end{eqnarray}
To estimate the left hand side, we observe that 
on $(\cup_i B(x_n^i,R'))^c$, $|\check{\rho_n}-1|\leq \min(|\rho_n-1|,\delta)$, therefore
\begin{eqnarray*}
\int_{\cup_i B(x_n^i,2^mR')\setminus B(x_n^i,2^{m-1}R')}-c_n(\check{\rho}_n-1)\dun \phi_n
+\frac{(\check{\rho}_n-1)|\nabla \phi_n|^2}{2} &&\\
\hspace{2cm}+\frac{1}{2}K'|\nabla\rho_n|^2(\check{\rho}_n-1)+g(\rho_n)(\check{\rho}_n-1)dx
&\lesssim& \frac{\widetilde{E_n}(\rho_n,\phi_n)}{k},
\end{eqnarray*}
Moreover $|\nabla \check{\rho_n}|\lesssim |\nabla\rho_n|+|\rho_n-1|$, therefore
% We deduce  $\|\check{\rho}_n-1\|_{L^\infty(\mathbb{T}_n^2)}\leq \delta$ and 
$$\displaystyle \int_{B(x_n^i,2^pR')\setminus B(x_n^i,2^{p-1}R')}|\nabla \check{\rho}_n|^2dx\lesssim 
\int_{B(x_n^i,2^pR')\setminus B(x_n^i,2^{p-1}R')}\widetilde{e}(\rho_n,\phi_n)dx,$$
so that the left hand side in \eqref{step1} is bounded by $E_n/k$.
This estimate, combined with $|\rho_n-1||\nabla \rho_n|^2\leq \delta |\nabla\rho_n|^2$ on 
$\mathcal{S}_n^k$ implies
\begin{eqnarray*}
\bigg|\int_{\mathcal{S}_n^k}-c_n(\rho_n-1)\dun \phi_n
 +K(\rho_n)|\nabla\rho_n|^2+g(\rho_n)(\rho_n-1)\,dx\bigg|\hspace{3cm}\\
 \leq C\bigg(\delta 
 \int_{\mathcal{S}_n^k}\widetilde{e}_ndx
 +\int_{\cup B(x_n^i,2^pR')\setminus B(x_n^i,2^{p-1}R')}\widetilde{e}_ndx\bigg).
\end{eqnarray*}
To conclude, we remark that near $\rho=1$, 
$\widetilde{g}(\rho)\sim \rho-1,\ \widetilde{G}(\rho)\sim (\rho-1)^2/2$,
so that $|\widetilde{g}(\rho_n)(\rho_n-1)-2\widetilde{G}(\rho_n)|\lesssim\delta 
\widetilde{G}(\rho_n)\leq \delta\widetilde{e}$. As a consequence
\begin{equation}\label{idmoment}
\bigg|\int_{\mathcal{S}_n^k}-c_np(\rho_n,\phi_n)+K|\nabla\rho_n|^2+2\widetilde{G}(\rho_n)\,dx\bigg|
 \leq C\bigg(\delta  \int_{\mathcal{S}_n^k}\widetilde{e}_ndx+
\frac{\widetilde{E}_n(\rho_n,\phi_n)}{k}\bigg)
\end{equation}
Putting together \eqref{idmasse} and \eqref{idmoment}, we find the expected result
\begin{equation*}
\bigg|\int_{\mathcal{S}_n^k}-c_np(\rho_n,\phi_n)
 +\frac{K(\rho_n)|\nabla\rho_n|^2+\chi(\rho_n)|\nabla\phi_n|^2}{2}+\widetilde{G}(\rho_n)\,dx\bigg|
 \leq C\bigg(\delta 
 \int_{\mathcal{S}_n^k}\widetilde{e}_ndx+\frac{\widetilde{E}_n}{k}\bigg).
\end{equation*}
\end{proof}

\noindent 
Now we combine these two lemmas to construct the sequence $R_n^k$ through a diagonal extraction.
\subsection{The  case \texorpdfstring{$c\geq 1$}{c}}
\paragraph{Construction of $R_n^1$} We recall that for any $n\geq 0$, $A_n^\delta\subset 
\cup_{i=1}^l B(x_n^i,1)$. We apply lemma 
\ref{coversepar} with $M=10$, this gives for any $n$ a subset $J_n^1\subset \{1,\cdots,l\}$ and 
$1\leq R_n^1\leq (20)^l$ such that 
$$\sqcup_{j\in J_n^1} B(x_n^j,R_n^1)\supset \cup_{i=1}^l B(x_n^i,1), \text{ and for any } 
(i,j)\in J_n^1, \ d(x_n^j,x_n^i)\geq 10 R_n^1.$$
We apply lemma \ref{lemmespread} with $k=1$\footnote{In this case obviously $p=1$, but this 
will not be the case in the rest of the induction argument.}, then \eqref{estimspread} is true on 
$\mathcal{S}_n^1=(\sqcup_{J_n^1} B(x_n^j,2R_n^1)\big)^c$. Since $(2R_n^1)_n$ and $|J_n^1|$ are 
bounded, there is an extraction $\psi_1(n)$ such that $2R_{\psi_1(n)}^1$ converges to some 
$\mathcal{R}^1\geq 2$ and $J_{\psi_1(n)}^1=J^1$ does not depend of $n$.
\paragraph{Construction of $R_n^2$} We apply once more lemma \ref{coversepar} to
$\cup_{i=1}^l B(x_{\psi_1(n)}^i,2)$ with $M=3\cdot 10$. For any $n$ there is a subset 
$J_{\psi_1(n)}^2\subset \{1,\cdots,l\}$, $2\leq R_{\psi_1(n)}^2\leq 2(60)^l$ such that 
\begin{eqnarray*}
\sqcup_{i\in J_{\psi_1(n)}^2} B(x_{\psi_1(n)}^i,R_{\psi_1(n)}^2)\supset 
\cup_{i=1}^l B(x_{\psi_1(n)}^i,1), \\
\text{ and for any } (i,j)\in J_{\psi_1(n)} ^1, 
\ d(x_{\psi_1(n)}^j,x_{\psi_1(n)}^i)\geq 30 R_{\psi_1(n)}^2.
\end{eqnarray*}
From lemma \ref{lemmespread} with $k=2$, for any $n$ there exists 
$1\leq m_{\psi_1(n)}^2\leq 2$ such that 
\eqref{estimspread} is true on $\mathcal{S}_{\psi_1(n)}^2=\big(\sqcup_{J_{\psi_1(n)}^2}
B(x_{\psi_1(n)}^j,2^{m_{\psi_1(n)}^2}R_{\psi_1(n)}^2)\big)^c$. Since $(2^{m_{\psi_1(n)}^2}
R_{\psi_1(n)}^2)_n$ and $|J_{\psi_1(n)}^2|$ are bounded, 
there is a sub-extraction $\psi_2(n)$ such that  $2^{m_{\psi_2(n)}^2}
R_{\psi_2(n)}^2\longrightarrow_n \mathcal{R}^2\geq 4$ and $J_{\psi_2(n)}^2=J^2$. \\
The generic argument at step $k$ to construct of $R_n^k$ is the following :
\paragraph{Construction of $R_n^k$} At step $k$, we have an extraction $\psi_{k-1}(n)$, we apply 
lemma \ref{lemmespread} to $\cup_{i=1}^l B(x_{\psi_{k-1}(n)}^i,2 ^k)$ with $M=10\cdot 3^{k-1}$, 
which gives again $2^k\leq R_{\psi_{k-1}(n)}^k\leq 2^k(20\cdot 3^{k-1})^l$, 
$J_{\psi_{k-1}(n)}^k\subset \{1,\cdots,l\}$ as before, then lemma \ref{lemmespread} 
provides $1\leq m_{\psi_{k-1}(n)}^k\leq 2^k$ such that \eqref{estimspread} is true on 
$\mathcal{S}_{\psi_{k-1}(n)}^k=\big(\sqcup_{J_{\psi_{k-1}(n)}^k}B(x_{\psi_{k-1}(n)}^j,
2^{m_{\psi_{k-1}(n)}^k}
R_{\psi_{k-1}(n)}^k)\big)^c$. The union is  disjoint since
\begin{equation*}
 d(x_{\psi_{k-1}(n)}^j,x_{\psi_{k-1}(n)}^i)\geq 10\cdot 3^{k-1}R_{\psi_{k-1}(n)}
 \geq 5\cdot(3/2)^{k-1} (2^{m_{\psi_{k-1}(n)}^k}R_{\psi_k(n)}^k).
\end{equation*} 
Since $(2^{m_{\psi_{k-1}(n)}^k}R_{\psi_{k-1}(n)}^k)_n$, $(|J_{\psi_{k-1}(n)}^k|)_n$ 
are bounded in $n$, there is an extraction 
$\psi_k$ such that 
$$
2^{m_{\psi_{k}(n)}^k}R_{\psi_{k}(n)}^k\longrightarrow_n\mathcal{R}^k\geq 2^{k+1},\ 
J_{\psi_k(n)}^k=J^k.
$$ 

\paragraph{Conclusion} Since $2\leq |J^k|\leq l$, there exists an extraction $\sigma$ such that 
$J^{\sigma(k)}=J$ does not depend on $k$ and $|J|\geq 2$. We consider the diagonal extraction 
$\psi_{\sigma(n)}(\sigma(n))=\Psi(n)$ and set for $n\geq k$, 
$\mathcal{R}_n^k:=2^{m_{\Psi(n)}^{\sigma(k)}}R_{\psi(n)}^{\sigma(k)}$, 
$(X_n^j)_{j\in J}:=(x_{\Psi(n)}^j)_{j\in J}$. By construction, 
$$
d(X_n^i,X_n^j)\geq 5\cdot(3/2)^{k-1}\mathcal{R}_n^k,\ 
\mathcal{R}_n^k\longrightarrow_n\mathcal{R}^k\geq 2^{\sigma(k)}
\longrightarrow_k+\infty,
$$
and for any $n\geq k$, according to lemma \ref{lemmespread}
\begin{equation*}
\bigg|\int_{(\sqcup_J B(X_n^j,\mathcal{R}_n^k))^c}\big(c_{\Psi(n)}p-\widetilde{e}\big)
(\rho_{\Psi(n)},\phi_{\Psi(n)})dx\bigg|\leq 
C\bigg(\delta \int_{(\sqcup_J B(X_n^j,\mathcal{R}_n^k))^c}\widetilde{e}dx
+\frac{\widetilde{E_{\Psi(n)}}}{\sigma(k)}
\bigg).
\end{equation*}
\subsection{The case \texorpdfstring{$c<1$}{c}}
In this case, for an arbitrary subset $\Omega$ we use the simple estimate :
\begin{equation*}
\forall\,x\in \Omega,\ |p|\leq \frac{(\rho-1)^2+\chi|\dun\phi|^2}{2\inf_{\Omega}\sqrt{\chi}}  
 \end{equation*}
Combining this with $G(\rho)=(1-\rho)^2/2+O((1-\rho)^3)$, this implies for $\delta$ small 
\begin{equation*}
\exists\,C>0:\forall\,x\in A_n^\delta,\ |p(\rho_n(x),\phi_n(x))|\leq 
\frac{\widetilde{e}(\rho_n(x),\phi_n(x))}{1-C\delta}.
\end{equation*}
For any set $\mathcal{S}\subset \mathcal{A}_n^\delta$, provided 
$n$ is large enough, $\delta$ small enough, we get
\begin{equation*}
\bigg|\int_{\mathcal{S}}\widetilde{e}-c_npdx\bigg|\geq \bigg(1-\frac{c_n}{1-C\delta}\bigg)
\int_{\mathcal{S}}\widetilde{e}dx\geq \frac{1-c}{2}\int_{\mathcal{S}}\widetilde{e}dx.
\end{equation*}
Now lemma \ref{lemmespread} with $k\geq 1$, $R=1$, $M=10\cdot 3^{k-1}$
provides $\mathcal{S}_n^k\subset \mathcal{A}_n^\delta$ on which equation \eqref{estimspread} combined 
with the inequality above implies for $\delta$ small enough
\begin{equation*}
\frac{1-c}{2}\int_{\mathcal{S}_n^k}\widetilde{e}dx\leq C\bigg(\frac{\widetilde{E}_n}{k}
+\delta \int_{\mathcal{S}_n^k}\widetilde{e}dx\bigg)\Rightarrow 
\int_{\mathcal{S}_n^k}\widetilde{e}\lesssim \frac{\widetilde{E}_n}{k}.
\end{equation*}
Therefore, arguing as for $c\geq 1$ we obtain extractions $\Psi,\sigma$ such that for $n\geq k$
\begin{equation*}
d(X_n^i,X_n^j)\geq 5\cdot(3/2)^{k-1}\mathcal{R}_n^k,\
\int_{\sqcup B(X_n^j,\mathcal{R}_n^k)}\widetilde{e}(\rho_{\Psi(n)},\phi_{\Psi(n)})dx,
\leq C\frac{\widetilde{E_n}}{\sigma(k)}
\end{equation*}
with $\lim_n \mathcal{R}_n^k=\mathcal{R}^k\geq 2^{\sigma(k)}$.

\section{Remarks on the one dimensional case}\label{unde}
The existence and stability of solitary waves for nonlinear Schr\"odinger type equations 
\begin{equation*}
i\partial_t\psi+\partial_x^2\psi=g(|\psi|^2)\psi,\ \text{with }g(\rho_0)=0,
\end{equation*}
is now quite well understood. Existence follows from basic 
ODE technics since the corresponding equation is integrable, stability is a more delicate 
issue, but can nevertheless be tackled in several ways. The first approach is to consider 
the minimization problem $\inf \{E_{NLS}(\psi),\ P_{NLS}(\psi)=p\}$. Due to better 
Sobolev embeddings in dimension $1$ it can be directly solved, the stability of minimizers 
then follows by the classical Cazenave-Lions \cite{CazLions} argument. This program has been 
carried at least in the Gross-Pitaevskii case $g(\rho)=\rho-1$ in \cite{BGS3}. More recently 
D. Chiron studied extensively in \cite{Chiron2} the stability and instability of traveling waves 
for very general $g(\rho)$. Among the variety of technics developed was an approach 
\`a la Grillakis-Shatah-Strauss which is very efficient in our case too.
In this section, we want to underline that traveling waves of \eqref{EK} and NLS share remarkable 
common features : 
\begin{enumerate}
\item their speed is bounded by the sound speed $c_s=\sqrt{\rho_0g'(\rho_0)}$ for \eqref{EK}, 
$\sqrt{2\rho_0g'(\rho_0)}$ for NLS,
\item if there exists a traveling wave of speed $c_0<c_s$, there exists a local branch of traveling waves 
parametrized by their speed as $\psi_c$ or $(\rho_c,\phi_c)$,
\item the stability criterion is $dP_{NLS}(\psi_c)/dc<0$, resp. $dP(\rho_c,\phi_c)<0$.
\end{enumerate}
The existence and conditional stability of solitary waves for \eqref{EK} in dimension one was 
already obtained in \cite{BDD2} with a stability criterion that can be easily proved as 
equivalent to $dP/dc<0$ (see remark \ref{stabaltern}). Nonlinear instability was left open, 
but using methods developed for Schr\"odinger type equations in \cite{Lin}, we will prove that 
$dP/dc>0$ implies nonlinear instability. This is the only new result of this section, which is 
structured as follows : we rewrite the equations in a more convenient form, and show the existence 
of traveling waves that can be parametrized by their speed (proposition \ref{existdim1}). Next we 
recall the stability criterion of \cite{GSS} and show that its assumptions are satisfied. Finally, 
we prove in theorem \ref{nonlininsta} that the failure of the stability criterion implies 
nonlinear instability. \vspace{1mm}\\ 
Let us now turn to the equations under study. We take $\rho_\infty>0$ 
and assume $g(\rho_\infty)=0$, $g'(\rho_\infty)>0$, we will study traveling waves with 
$\lim_{\pm\infty} \rho=\rho_\infty$. As in the rest of the article, we assume that $g$ and $K$ are 
smooth on $]0,+\infty[$ in order to avoid  technical issues. $G$ is the primitive of $g$ that cancels 
at $\rho_\infty$.
In order to avoid the peculiar space $\dot{H}^1$, we will use a slight modification of the hamiltonian
and momentum. Instead of
$$E(\rho,\phi)=\int_{\R} \frac{\rho|\nabla \phi|^2+K|\nabla \rho^2|}{2}+G(\rho)dx,$$ 
defined for $(\rho,\phi)\in H^1\times \dot{H}^1$, we consider 
\begin{equation*}
E(\rho,u)=\int_{\R} \frac{\rho|u|^2+K|\nabla \rho^2|}{2}+G(\rho)dx,\ 
P(\rho,u)=\int_{\R} (\rho-\rho_\infty)udx.
\end{equation*}
defined for $(\rho,u)\in (\rho_\infty+H^1)\times L^2$ with $\rho>0$. For the variables $(\rho,u)$, 
the Euler-Korteweg system has the following hamiltonian structure 
\begin{equation}\label{alterham}
\partial_t\begin{pmatrix}
\rho\\ u\end{pmatrix}
=\begin{pmatrix}
0 & -\partial_x\\
  -\partial_x & 0
 \end{pmatrix}
\begin{pmatrix}
 \frac{\delta E}{\delta \rho}\\
 \frac{\delta E}{\delta u}
\end{pmatrix}
=J\delta E.
\end{equation}
Traveling waves of speed $c$ can be seen as critical points of $E-cP$ : 
if $\rho(x-ct)$, $u(x-ct)$ solves \eqref{alterham} with 
$\lim_{\pm\infty}\rho= \rho_\infty,\ \lim_{\pm\infty}u=0$, then 
\begin{equation*}
\left\{
\begin{array}{ll}
-c(\rho-\rho_\infty)+\rho u=0\\
-cu+u^2/2+g(\rho)=K\rho''+\frac{1}{2}K'(\rho')^2
\end{array}
\right.
\Leftrightarrow  
c\begin{pmatrix}
\frac{\delta P}{\delta \rho}\\
\frac{\delta P}{\delta u}
\end{pmatrix}
=
\begin{pmatrix}
 \frac{\delta E}{\delta \rho}\\
 \frac{\delta E}{\delta u}
\end{pmatrix}
\end{equation*}
Obviously if $(\rho,u)$ is a traveling wave of speed $c$, $(\rho,-u)$ is a traveling wave of speed $-c$, 
therefore we focus on the case $c>0$ (we choose not to consider the degenerate case $c=0$).
This ODE system can be elementarily integrated: from the first equation, $u=c(\rho-\rho_\infty)/\rho$, 
injecting this in the second equation, and
multiplying it by $\rho'$, we obtain after integration
\begin{equation}\label{eqdiffrho}
\frac{-c^2}{2\rho}(\rho-\rho_\infty)^2+G(\rho)=\frac{1}{2}K(\rho')^2,
\end{equation}
Letting $x\rightarrow \infty$, we find 
\begin{equation*}
0\leq \frac{1}{2}K(\rho')^2=\frac{(\rho-\rho_\infty)^2}{2\rho_\infty}(
\rho_\infty g'(\rho_\infty)-c^2)+O(\rho-\rho_\infty)^3.
\end{equation*}
We deduce the so-called subsonic condition
$$|c|\leq \sqrt{\rho_\infty g'(\rho_\infty)}:=c_s.$$ 
Conversely, if $0<c<\sqrt{\rho_\infty g'(\rho_\infty)}$ 
consider the application 
$$F(\rho)=\frac{-c^2}{2\rho}(\rho-\rho_\infty)^2+G(\rho).$$ 
On a neighbourhood of $\rho_\infty$, $F>0$, and since $\lim_{0^+}F_c(\rho)=-\infty$ we can define
$\rho_m=\sup\{\rho<\rho_\infty: F(\rho)=0\}>0$. The set $\{(\rho,\rho'):\ \rho'=\pm \sqrt{F(\rho)},\ 
\rho\in [\rho_m,\rho_\infty]\}$ forms a homoclinic orbit of the differential equation 
\eqref{eqdiffrho} under the (generically true) condition $F'(\rho_m)> 0$. If 
$F'(\rho_c)=0$ the set corresponds to two heteroclinic profiles (of infinite energy). Symmetrically, 
if $\rho_M=\inf\{\rho>\rho_\infty:\ F(\rho)=0\}$ is finite, the set $\{(\rho,\rho'):\ \rho'=\pm \sqrt{F(\rho)},\ 
\rho\in [\rho_\infty,\rho_M]\}$ forms a homoclinic orbit if $F'(\rho_M)<0$. We also point out the identity 
\begin{equation}\label{nonzeroP}
P(\rho_c,u_c)=c\int_{\R} \frac{(\rho_c-\rho_\infty)^2}{\rho_c}dx,
\end{equation}
so that for any traveling wave with non zero speed, $P(\rho_c,u_c)\neq 0$.
\\
Finally, consider $F$ as a function of $(\rho,c)$. Given $0<c_0<c_s$, the condition $F(\rho_m,c_0)=0,\ 
\partial_{\rho}F(\rho_m,c_0)\neq 0$ implies from the implicit function theorem there exists 
$\rho_I(c)$
smooth, defined on a neighbourhood of $c_0$ and a neighbourhood of $(\rho_m,c)$ such that 
$F(\rho,c)=0$  iff $\rho=\rho_I(c)$. Up to shrinking the neighbourhood of $c$, 
$\rho_I(c)=\sup \{\rho<\rho_m:\ F(\rho,c)=0\}$, by continuity 
$\partial_{\rho}F(\rho_I(c),c)\neq 0$, 
and in particular this gives a small 
\emph{branch} of solitary waves parametrized by $c$, that have for minimal value 
$\rho_I(c)$.
These observations can be summarized with the following proposition.
\begin{prop}\label{existdim1}
There exists no nontrivial traveling wave for $c>c_s$. For $0< c<c_s$, 
there exists a nontrivial traveling wave 
if and only if at least one of the two cases is true
\begin{itemize}
 \item There exists $\rho_m<\rho_\infty$ such that $F>0$ on $(\rho_m,\rho_\infty)$, 
 $F(\rho_m)=0$, $F'(\rho_m)>0$. 
 In this case, up to translation $\rho$ is the solution of the Cauchy 
 problem 
 \begin{equation*}
 \left\{
 \begin{array}{ll}
\displaystyle \frac{1}{2}K'(\rho')^2+K\rho''=\frac{-c^2(\rho^2-\rho_\infty)^2}{\rho^2} ,\\
\rho(0)=\rho_m,\ \rho'(0)=0.
 \end{array}\right.
\end{equation*}
It is even, decreasing on $]-\infty,0]$.
\item There exists $\rho_M>\rho_\infty$ such that $F>0$ on $]\rho_\infty,\rho_M[$, 
 $F(\rho_M)=0$, $F'(\rho_M)<0$.  In this case, up to translation $\rho$ is the solution of the Cauchy 
 problem 
 \begin{equation*}
 \left\{
 \begin{array}{ll}
\displaystyle \frac{1}{2}K'(\rho')^2+K\rho''=\frac{-c^2(\rho^2-\rho_\infty)^2}{\rho^2} ,\\
\rho(0)=\rho_M,\ \rho'(0)=0.
 \end{array}\right.
\end{equation*}
It is symmetric, increasing on $]-\infty,0]$.
\end{itemize}
In both cases, $P(\rho,u)>0$.
Moreover, near any traveling wave of speed $c_0<c_s$ there exists a branch of traveling waves 
that can be parametrized by $c\in (c_0-\varepsilon,c_0+\varepsilon)$ for $\varepsilon$ small enough.
\end{prop}
Given a branch of traveling waves defined on some interval of speeds $I$, we abusively denote 
$E(c),\ P(c)$ the energy and momentum of the traveling wave of speed $c$ in this branch, $E',P'$ 
their derivative with respect to $c$.
Regarding stability, following the famous result of Grillakis-Shatah-Strauss \cite{GSS}, 
the \emph{moment of instability} was defined in \cite{Benzoni5} as 
\begin{equation*}
m(c)=E(c)-cP(c).
\end{equation*}
Let us shortly summarize the framework from \cite{GSS}: the Euler-Korteweg equations are seen as the 
hamiltonian system \eqref{alterham}, it is invariant by translation, the conservation law 
associated to the translation invariance is the momentum $P(\rho,u)$. 
% Indeed the infinitesimal 
% generator of the translation group is $\partial_x$, and we have the trivial identity 
% \begin{equation}\label{momenttrans}
%  -\partial_x 
%  \begin{pmatrix}
%   \rho \\ u
%  \end{pmatrix}
% =
% \begin{pmatrix}
%  0 & -\partial_x\\
%  -\partial_x
% \end{pmatrix}
% \begin{pmatrix}
%  \delta P/\delta \rho\\ 
%  \delta P/\delta u
% \end{pmatrix}
% :=J\delta P.
% \end{equation}
% Denoting $\langle\cdot,\cdot \rangle$ the $(L^2)^2$ scalar product, this 
% implies for any smooth solution of \eqref{alterham}
% \begin{equation*}
% \frac{d}{dt}P(\rho,u))=\langle \delta P,(\partial_t\rho,\partial_tu)\rangle =
% \langle \delta P, J\delta E\rangle =\langle \partial_x(\rho,u),\delta E\rangle =\frac{d}{ds}
% E(\rho(\cdot+s),u(\cdot+s))|_{s=0}=0.
% \end{equation*}
Since a traveling wave satisfies $\delta E-c\delta P=0$, it is a critical point of 
$E-cP$. \\
We say that a traveling wave is conditionally orbitally stable if for any $\varepsilon>0$, there 
exists $\delta >0$ such that if $\|(\rho_0,u_0)-(\rho_c,u_c)\|_{H^1\times L^2}<\delta$ and the solution exists
on $[0,T)$ then 
\begin{equation*}
\sup_{t\in [0,T)}\inf_{y\in \R} \|(\rho(t,\cdot+y),u(t,\cdot+y))
-(\rho_c,u_c)\|_{H^1\times L^2}<\varepsilon.
\end{equation*}
\begin{theo}[\cite{GSS}]\label{stabsol}
Under the following assumptions: 
\begin{itemize} 
 \item $\delta^2E-c\delta^2P$ has only one negative simple eigenvalue
 \item its kernel is spanned by $\partial_x(\rho_c,u_c)$, the rest of its spectrum is positive 
 bounded away from $0$
 \item $J$ is onto
 \end{itemize}
then the traveling wave of speed $c$ is conditionally orbitally stable if and only 
if $m''(c)>0$. If $J$ is not onto
the ``if'' part remains true, but the ``only if'' part may fail.
\end{theo}
\begin{rmq}\label{stabaltern}
An alternative version of $m''(c)>0$ can be stated as follows: since any traveling 
wave of speed $c$ is a critical point of the functional $(\rho,u)\mapsto E-cP$, 
we have for any $c\in I$, $E'(c)-cP'(c)=0$ , differentiating twice $E(c)-cP(c)$, we find 
\begin{equation*}
m''(c)=-P'(c),
\end{equation*}
so that $m''>0$ is equivalent to $P'<0$. In this case the application 
$c\rightarrow P(c)$ is locally invertible and we may parametrize 
$E$ by $P$. Since $dE/dP=E'/P'=c$, we have 
\begin{equation*}
\frac{d^2E}{dP^2}=\frac{dc}{dP}<0,
\end{equation*}
so that the stability condition implies the strict concavity of $E(P)$. We point out that in 
dimension $2$ the curve $\widetilde{E}_{\min}(P)$ is concave (proposition \ref{concE}). 
This is an indication in favour of the stability of the traveling waves that we constructed.
\end{rmq}

\paragraph{Notations:} $(\rho_c,u_c)$ is a branch of traveling waves locally parametrized by their speed $c$.
We denote $\langle \cdot,\cdot \rangle$ for both the $L^2$ and $(L^2)^2$ 
scalar product. We use the variable $r=\rho-\rho_\infty$, set 
$r_c:=\rho_c-\rho_\infty$ and set with an abusive notation $P(r,u):=P(\rho,u)$, then
\begin{equation*}
P(r,u)=\int ru\,dx,\ \delta P(r,u)=
\begin{pmatrix}
 u\\ r
\end{pmatrix}=
\begin{pmatrix}
0 & 1 \\
1 & 0
\end{pmatrix}
\begin{pmatrix}
 r\\ u
\end{pmatrix}.
\end{equation*}
We denote $\mathcal{L}=\delta^2E-c\delta^2P$, $U_c=(r_c,u_c)$.\\
For any function depending on the speed $f_c$ (and possibly on the $x$ variable), 
we denote $f'_c:=df_c/dc$. To avoid confusion we denote $\partial_x$ the spatial derivative.

\paragraph{Spectral assumptions for $\mathcal{L}:=\delta^2E-c\delta^2P$ }
They were obtained in \cite{BDD2} for the lagrangian formulation of the equations. 
The argument in the eulerian variable is slightly more involved, we include it for 
completeness:
\begin{equation}
\mathcal{L}(\rho_c,u_c)=
\begin{pmatrix}
\mathcal{M} & u_c-c\\
u_c-c & \rho_c
\end{pmatrix},\
\mathcal{M}r= \bigg(G''-\frac{K''(\partial_x\rho_c)^2}{2}-K'\partial_x^2\rho_c\bigg)r
-\partial_x(K\partial_xr).
\end{equation}
Due to the invariance by translation, we have $(\delta E-c\delta P)(\rho_c,u_c)(\cdot+x)=0$,
by differentiation in $x$ we get
$\mathcal{L}(\rho_c,u_c)\partial_x(\rho_c,u_c)=0$. 
Conversely if $U=(U_1,U_2)\in \text{Ker}(\mathcal{L})$, we have $U_2= \frac{c-u_c}{\rho_c}U_1$, and 
$U_1\in \text{Ker}(\mathcal{M}-(u_c-c)^2/\rho_c)$. 
As $\mathcal{M}-(u_c-c)^2/\rho_c$ is a 
Sturm-Liouville type operator, its kernel is of dimension one and since $\partial_x\rho_c\in 
\text{Ker}(\mathcal{M}-(u_c-c)^2/\rho_c)$, there exists 
$\lambda \in \R$ such that $U_1=\lambda \partial_x\rho_c$. Next using 
$u_c=c(1-\rho_\infty/\rho_c)$ 
$$U_2=\lambda\frac{c-u_c}{\rho_c}
\partial_x\rho_c=\frac{\lambda c\rho_\infty}{\rho_c^2}\partial_x\rho_c=\lambda \partial_xu_c,$$ 
so $\partial_x(\rho_c,u_c)$ spans $\text{Ker}(\delta^2E-c\delta^2P)$.\\
Furthermore as $\partial_x\rho_c$ has exactly one zero, 
from Sturm Liouville's theory the operator $\mathcal{M}-\frac{(u_c-c)^2}{\rho_c}$ has 
exactly one negative eigenvalue. In particular, if $r_-$ is an eigenvector 
associated to the negative eigenvalue and $U_-=(r_-,-(u_c-c)r_-/\rho_c)$, then
\begin{equation}\label{defumoins}
\langle\mathcal{L}U_-,U_-\rangle=\langle (\mathcal{M}-(u_c-c)^2/\rho_c)r_-,
r_-\rangle<0,
\end{equation}
so that $\delta^2E-c\delta^2P$ has at least one negative eigenvalue. Conversely, 
if $\lambda<0$ is an eigenvalue of $\delta^2E-c\delta^2P$ with eigenvector $(U_1,U_2)$, 
from basic computations 
\begin{equation*}
\bigg(\mathcal{M}-\frac{(u_c-c)^2}{\rho_c}-\frac{\lambda(u_c-c)^2}{\rho_c(\lambda-\rho_c)}\bigg)
U_1=\lambda U_1,
\end{equation*}
so that $\lambda$ is an eigenvalue of $\delta^2E-c\delta^2P$ if and only if it is an 
eigenvalue of 
$$\mathcal{M}_\lambda= \mathcal{M}-\frac{(u_c-c)^2}{\rho_c}
-\frac{\lambda(u_c-c)^2}{\rho_c(\lambda-\rho_c)}.$$
As the application $\lambda\in \R^-\rightarrow \lambda/(\lambda-\rho_c)$ is decreasing, the 
family $\mathcal{M}_{\lambda}$ is decreasing too (in the sense of the scalar product). 
Let $\lambda_-<0$ be the minimal eigenvalue of $\delta^2E-c\delta^2P$. Since 
$\mathcal{M}_0=\mathcal{M}-(u_c-c)^2/\rho_c$, it has only one negative eigenvalue, and thus 
so does $\mathcal{M}_\lambda$ for $\lambda_-\leq \lambda\leq 0$. If $\delta^2E-c\delta^2P$
had an other negative eigenvalue $\lambda_-<\lambda'<0$, then $\lambda'$ would be the only
negative eigenvalue of $\mathcal{M}_{\lambda'}$. By monotony $\lambda'<\lambda_-$ which 
is absurd.
\\
For the last condition, we have characterized the negative eigenvalue and the kernel.
It suffices then to observe that thanks to the subsonic condition 
\begin{equation*}
\lim_{x\rightarrow \infty } G''(\rho_c)-\frac{K''(\rho_c)(\partial_x\rho_c)^2}{2}
-K'(\rho_c)\partial_x^2\rho_c-\frac{(u_c-c)^2}{\rho_c}=\frac{\rho_\infty g'(\rho_\infty)-c^2}
{\rho_\infty}
>0,
\end{equation*}
thus the essential spectrum of $\mathcal{M}-\frac{(u_c-c)^2}{\rho_c}$ is 
positive bounded away from zero.\vspace{2mm}\\
Theorem \ref{stabsol} can now be applied :
\begin{coro}[orbitaly stability, \cite{BDD2}]
If $-P'(c)=m''(c)>0$, then $(\rho_c,u_c)$ is conditionally orbitally stable.
\end{coro}
% 
% \begin{rmq}
% Intuitively, the stability condition may be seen as a way to overcome the lack of 
% coercivity of $U\rightarrow E(U)-E(U_c)$. Consider its restriction on $\{P(U)=P(U_c)\}$. 
% On this set,
% it is equivalent to have the coercivity of $(E-cP)(U)-(E-cP)(U_c)$. Since 
% \begin{equation*}
% 0<m''(c)=-\langle \mathcal{L}U'_c,U'_c\rangle,
% \end{equation*}
% $E-cP$ decays in the direction $U'_c$. On the other hand, for any direction $y$ such that 
% $\langle \delta P(U_c),y\rangle =0$, we have
% $\delta(E-cP)(U_c)=0\Rightarrow \langle \mathcal{L}U'_c,y\rangle=0$ so that the hyperplane 
% tangent to $\{P(U)=P(U_c)\}$ at $U_c$ is (somehow) transverse to the decrasing direction 
% $U'_c$, hence (neglecting the kernel of $\mathcal{L}$) $U_c$ is a local minimum 
% of $(E-cP)|_{\{P(U)=P(U_c)\}}$.
% \end{rmq}
\begin{rmq}
Unfortunately, the well-posedness theory from \cite{BDD} only provides local existence for 
$(\rho(t=0),u(t=0))\in (\rho_0+H^{s+1})\times H^s,\ s>3/2$, therefore it is not clear if a smooth 
solution starting near a traveling wave exists for all times. At least in the case $K=1/\rho$, one can 
combine the existence of global solutions to NLS that remain bounded away from $0$ and use the 
Madelung transform to convert them into solutions of \eqref{EK}.
\end{rmq}

\begin{rmq}
The condition $P'(c)<0$ seems a bit easier to check than $m''>0$. For example for 
$\rho_m<\rho_\infty$ from \eqref{eqdiffrho}
\begin{equation*}
P(\rho_c,u_c)=\int_{\R}(\rho_c-\rho_\infty)u_c\,dx=
2\int_{\rho_m}^{\rho_\infty}\frac{c(\rho-\rho_\infty)^2}{\rho}\sqrt{\frac{K}{2(G-\frac{c^2}{2\rho}
(\rho-\rho_\infty)^2}))}d\rho,
\end{equation*}
with $\rho_m$ the first zero of $G-\frac{c^2}{2\rho}(\rho-\rho_\infty)^2$ below $\rho_\infty$.
\end{rmq}

Nonlinear instability is not a direct application of theorem \ref{stabsol}, indeed
$\begin{pmatrix}
0 & -\partial_x\\
-\partial_x & 0
\end{pmatrix}$  is not onto so the only if part can not be used. 
Of course there is no gain in adopting
the formulation with $(\rho,\phi)\in H^1\times\dot{H}^1$ : in this case $J=
\begin{pmatrix}
0& -1\\
1 & 0
\end{pmatrix}$, but $\dot{H}^1$ is not a Hilbert space.
Nevertheless this obstruction was overcome in various settings, in particular we shall 
follow the approach of Lin \cite{Lin} (see also \cite{BonSS}) to prove the following result:
\begin{theo}\label{nonlininsta}
Let $(\rho_c,u_c)$ be a traveling wave of speed $c>0$. If $\frac{dP}{dc}>0$, then the traveling 
wave is unstable, i.e., there exists $\varepsilon>0$ such 
that for any $\delta>0$, there exists $(\rho_0,u_0)\in H^3\times H^2$ such that 
$\|\rho_0-\rho_c\|_{H^1}+\|u_0-u_c\|_{L^2}<\delta$ and either the corresponding solution $(\rho,u)$ 
blows up in finite time, or 
\begin{equation*}
\sup_{t\in \R^+}\inf_{y\in \R}\|\rho(t,\cdot+y)-\rho_c\|_{H^1}+\|u(t,\cdot+y)-u_c\|_{L^2}\geq 
\varepsilon.
\end{equation*}
\end{theo}
\noindent We recall the notation $U_c=(r_c,u_c)=(\rho_c-1,u_c)$. 
The proof in the framework of \cite{GSS} relies on the existence of a smooth curve 
$\psi(s):\ (-\eta,\eta)\longrightarrow H^1\times L^2$ for some $\eta>0$, with
\begin{equation*}
\psi(0)=(r_c,u_c),\ P(\psi(s))=P(r_c,u_c),\ 
\langle (\delta^2E-c\delta^2P)\psi'(0),\,\psi'(0)\rangle <0.
\end{equation*}
It provides an ``unstable direction'' $y=d\psi/ds|_{s=0}$ such that 
\begin{equation}\label{propy}
\langle \delta^2(E-cP)y,\, y\rangle<0,\ \langle \delta P(U_c),\,y\rangle=0,
\end{equation}
and a Lyapunov function $A(U)=\langle -J^{-1}y,\, U(\cdot+x_{\min}(U))\rangle$, where $x_{\min}(U)$ 
minimizes $\|(r_c,u_c)-U(\cdot+x)\|_{H^1\times L^2}$ (see lemma \ref{opttrans} below). For 
$0<s<<1$, it is proved that the solution $(r(t),u(t))$ with Cauchy data $(r(0),u(0))=\psi(s)$ 
is unstable due to some growth of $A(r(t),u(t))$. This approach raises two issues: 
\begin{itemize}
\item $J^{-1}y$ does not exist a priori. The method in \cite{Lin} is to construct $y_1\in 
\text{range}(J)$ close to $y$, which still satisfies \eqref{propy}, and carry on the proof. 
\item All constructions are performed in the natural 
functional settings $(r,u)\in H^1\times L^2$, but the best local well-posedness result requires 
$(r(0),u(0))\in H^{s+1}\times H^s$, $s>3/2$ (see \cite{BDD}). We use a
density argument to replace the  unstable initial data $\psi(s)\in H^1\times L^2$ by a regularized version.
\end{itemize}
\noindent This program requires a collection of lemmas that we prove only when there 
is a significant difference with \cite{GSS}. 

\begin{lemma}[lemma $3.2$  \cite{GSS}.]\label{opttrans}
Let $$V_{\varepsilon}=\{(r,u)\in H^1\times L^2:\ \inf_x(\|r(\cdot+x)-r_c\|_{H^1}+
\|u(\cdot+x)-u_c\|_{L^2})<\varepsilon.$$
For $\varepsilon$ small enough, there exists a smooth map 
$x_{\min}:\ V_{\varepsilon}\rightarrow \R$ which realises the inf, namely :
\begin{equation*}
\|U(\cdot+x_{\min}(U))-U_c\|_{H^1\times L^2}
=\inf_x\|U(\cdot+x)-U_c\|_{H^1\times L^2}.
\end{equation*}
 Moreover $x_{\min}(U(\cdot+r))=x_{\min}(U)-r$.
% \partial_xU_c\big)_{H^1\times L^2}=0$
\end{lemma}
The following lemma is the only one where the lack of surjectivity of $J$ 
requires some corrections.
\begin{lemma}[theorem 4.1 \cite{GSS}]\label{instacurve}
There exists $y\in \text{Im}(J)\cap (H^1)^2$ such that 
\begin{equation*}
\langle\delta P(U_c),y\rangle=0,\ \langle \mathcal{L}y,y\rangle<0,
\end{equation*}
and a smooth curve $\psi:(-\eta,\eta)\rightarrow \{(U\in H^1\times L^2:\ P(U)=P(U_c)\}$
with 
$$\frac{d\psi}{ds}(0)=y,\ \frac{d^2\psi}{ds^2}(0)<0.$$
In particular, $s=0$ is a local maximum of $E(\psi(s))$.
\end{lemma}
\begin{proof}
Let $U_-$ as in \eqref{defumoins}, $y_0=\alpha U'_c+U_-$, $\alpha=-\langle \delta P(U_c),
U_-\rangle/\langle \delta P(U_c),U_c'\rangle$. We have $\langle \delta P(U_c),y_0\rangle=0$, moreover
$\delta E(U_c)-c\delta P(U_c)=0$, by differentiation in $c$, $\mathcal{L}U'_c=\delta P(U_c)$.
This implies 
\begin{eqnarray*}
\langle \mathcal{L}y_0,y_0\rangle =\alpha \langle\mathcal{L}U'_c,y_0\rangle 
+\langle\mathcal{L}U_-,y_0\rangle&=&
\alpha \langle\delta P(U_c),y_0\rangle+\langle\mathcal{L}U_-,y_0\rangle\\
&=& \langle \mathcal{L}U_-,U_-\rangle +\alpha \langle \delta P(U_c),U_-\rangle\\
&=& \langle \mathcal{L}U_-,U_-\rangle-\frac{\big(\langle \delta P(U_c),U_-\rangle\big)^2}
{\langle \delta P(U_c),U'_c\rangle} 
\end{eqnarray*}
From \eqref{defumoins}, $\langle \mathcal{L}U_-,U_-\rangle<0$, $\langle \delta P(U_c),
U'_c\rangle=P'(c)>0$, thus 
$\langle \mathcal{L}y_0,y_0\rangle<0$.\\
We construct then $y\in \text{Im}(J)$ close to $y_0$.
From classical ODE arguments $U_c(x)$ and $U_-(x)$ converge exponentially fast to $0$ 
at infinity, in particular $(1+|x|)y_0(x)\in L^1$. According to \cite{Lin}, lemma $5.2$, for any 
$\mu>0$ there exists $d=(d_1,d_2)\in (H^1)^2,\ \|d\|_{H^1}<\mu$,  such that 
$$(1+|x|)d\in L^1,\ \int_{\R} y_0+d\, dx=0,\ \int_{\R} d_1u_c\,dx=\int_{\R}d_2r_c\,dx=0.$$
In particular $\langle\delta P(U_c),d\rangle=\int d_1u_c+d_2r_c\,dx=0$. Let us set 
$y=y_0+d$. Then by construction $\langle\delta P(U_c),y\rangle=0$. For $\mu$ small enough
\begin{equation*}
\langle \mathcal{L}y,y\rangle<0,\ \text{moreover }
\int_{-\infty}^xy(s)ds=O(1/(1+|x|)),
\end{equation*}
so that $J^{-1}y$ is well defined and belongs to 
$(H^2)^2$. Now 
since $\delta P(U_c)\neq 0$, $E:=\{U:\ \langle \delta P(U_c), U\rangle =0\}$ is a closed hyperplane of 
$H^1\times L^2$ with $y\in E$. By the implicit function theorem there exists a 
neighbourhood $\mathcal{U}
\subset E$ and an application $F:\mathcal{U}\rightarrow H^1\times L^2$ such that for $e\in \mathcal{U}$, 
$P(U_c+F(e))=P(U_c)$, $\delta F(0)=I_d$. In particular, if we set for $s$ small enough 
$\psi(s)=U_c+F(sy)$ we obtain 
$P(\psi(s))=P(U_c),\ \displaystyle \frac{d\psi}{ds}\bigg|_0=y$.
% \begin{eqnarray*}
% P(\psi(s))=P(U_c),\ \frac{dE(\psi(s))}{ds}|_{s=0}&=&\langle (\delta E-c\delta P)(U_c),y\rangle=0,\\ 
% \frac{d^2E(\psi(s))}{ds^2}|_{s=0}&=&\langle \mathcal{L}y,y\rangle <0.
% \end{eqnarray*}
Using $E(\psi(s))=(E-cP)(\psi(s))+cP(U_c)$ we have 
\begin{eqnarray*}
\frac{dE(\psi(s))}{ds}|_{s=0}=\langle (\delta E-c\delta P)(U_c),y\rangle=0,\
\frac{d^2E(\psi(s))}{ds^2}|_{s=0}=\langle \mathcal{L}y,y\rangle <0.
\end{eqnarray*}
\end{proof}

\noindent The next lemmas correspond to \cite{GSS} from lemma $4.2$ to lemma $4.6$. 
Let $y=(y_1,y_2)$ from lemma \ref{instacurve} and define $Y:=-\int_{-\infty}^x(y_2,y_1)\in (H^2)^2$ so that $JY=y$.
\begin{lemma}
 The map $\displaystyle A:\ U\in V_\varepsilon \rightarrow A(U)=\langle -Y,U(\cdot+x_{\min}(U))
 \rangle$ is $C^1$ and  satisfies 
 \begin{equation*}
\forall\,U\in V_\varepsilon,\ J\delta A(U)\in H^1\times L^2,
J\delta A(U_c)=-y,\ \langle \delta P(U),\ J\delta A(U)\rangle =0.
 \end{equation*}
\end{lemma}
\begin{lemma}
The differential equation $U'(\lambda)=-J\delta A(U(\lambda)),\ U(0)=U_0\in V_\varepsilon$ 
defines a local flow in $V_\varepsilon$, denoted $R(\lambda,U_0)$. It satisfies 
\begin{eqnarray}
R(\lambda,U_0)(\cdot+s)=R(\lambda,U_0(\cdot+s)),\\
\frac{d}{d\lambda}P(R(\lambda,U_0))=\langle \delta P\big(R(\lambda,U_0)\big),\, 
-J\delta A\big(R(\lambda,U_0)\big)\rangle =0,\\
\frac{dR(\lambda,U_c)}{d\lambda}\bigg|_{0}=-J\delta A(U_c)=y.
\end{eqnarray}
 \end{lemma}
\begin{lemma}
Let $M(U):=U(\cdot+x_{\min}(U))$, $U_{-1}$ an eigenvector associated to the negative eigenvalue of 
$\delta^2E-c\delta^2P$.
The solutions of $\langle M(R(\lambda,U))-U_c,U_{-1}\rangle=0$ can be parametrized 
as $(\Lambda(U),U)$, where $\Lambda$ is a functional in $C^1(V_\varepsilon,\R)$. For 
$U\in V_\varepsilon$ such that $P(U)=P(U_c)$,
% For any $U$ such that 
% $P(U)=P(U_c)$ and $U$ is not a translate of $U_c$, 
% \begin{equation*}
% E\big(R(\Lambda(U),U)\big)>E(U_c).
% \end{equation*}
\begin{equation}\label{ineqEner}
 E(U_c)<E(U)+\Lambda(U) Q(U),\ Q=\langle \delta E,-J\delta A\rangle.
\end{equation}

\end{lemma}

% \begin{lemma}\label{ineqEner}
% For any $U \in V_\varepsilon$ such that $P(U)=P(U_c)$, 
% \begin{equation}
%  E(U_c)<E(U)+\Lambda(U) Q(U),\ Q=\langle \delta E,-J\delta A\rangle.
% \end{equation}
% \end{lemma}

\begin{lemma}
For $\psi:\ (-\eta,\eta)\rightarrow V_\varepsilon$ given in lemma \ref{instacurve}, 
$Q(\psi(s))$ changes sign at $0$.
\end{lemma}

\subparagraph{End of proof of theorem \ref{nonlininsta}} Since $Q$ changes sign, there exists 
$s$ such that $Q (\psi(s))>0$. Since 
$\lim_0\psi(s)=U_c$, $\psi(s)$ can be chosen arbitrarily close to $U_c$. From lemma \ref{instacurve}
$E(\psi(s))<E(U_c)$ and for $s$ small enough 
using \eqref{nonzeroP} we have $P(\psi(s))>0$. For $(\varphi_n)_{n\geq 0}$ a 
standard sequence of mollifiers, 
$$\|\varphi_n*\psi(s)-\psi(s)\|_{H^1\times L^2}\rightarrow 0,\ \varphi_n*\psi(s)\in H^3\times H^2.$$
For $n$ large we can assume $P(\varphi_n*\psi(s))\neq 0$, $E(\varphi_n*\psi(s))<E(U_c)$ and we define 
$$ U_n=\sqrt{\frac{P(\psi(s))}{P(\varphi_n*\psi(s))}}\varphi_n*\psi(s),$$ 
As $P(\psi(s))/P(\varphi_n*\psi(s))\longrightarrow_n1$, for $n$ large enough 
$E(U_n)<E(U_c)$ and by construction $P(U_n)=P(U_c)$. 
Let $U=(r,u)(t)$ the solution of \eqref{alterham} with initial data $U_n$. 
By conservation of $E$ and $P$ (see \cite{BDD}), and \eqref{ineqEner}, as long as $U(t)$ remains in $V_\varepsilon$ 
\begin{equation*}
E(U(t))=E(U_n)<E(U_c),\ E(U_c)<E(U(t))+\Lambda(U(t))Q(U(t)).
\end{equation*}
This implies $\Lambda>0$ and up to diminishing $\varepsilon$ we can assume $\Lambda\leq 1$, 
so that $Q(U(t))\geq E(U_c)-E(U_n)>0$. Then if $U(t)\in V_\varepsilon$,
\begin{equation*}
A(U(t))\leq \|Y\|_2(\|U_c\|_2+\varepsilon),\
\frac{d}{dt}A(U(t))=\langle J\delta E,\delta A\rangle =Q(U(t))\geq E(U_c)-E(U_n),
\end{equation*}
which can only remain true for a finite time. Thus $U(t)$ must exit $V_\varepsilon$ or blows up
before.
\bibliography{biblio}
\bibliographystyle{plain}

\end{document}